\documentclass[reqno]{amsart}
\usepackage{amssymb, amsfonts}
\usepackage[usenames, dvipsnames]{color}
\usepackage[hypertex]{hyperref}

\usepackage{verbatim}

\numberwithin{equation}{section}

\newtheorem{theorem}{Theorem}[section]

\newtheorem{lemma}[theorem]{Lemma}
\newtheorem{proposition}[theorem]{Proposition}
\newtheorem*{AssumpA}{Assumption A $(\rho,\varepsilon)$}
\newtheorem*{AssumpA'}{Assumption A$'$ $(\rho,\varepsilon)$}
\newtheorem*{AssumpA''}{Assumption A$''$ $(\rho,\varepsilon)$}

\theoremstyle{definition}
\newtheorem{remark}[theorem]{Remark}

\theoremstyle{definition}

\theoremstyle{definition}
\newtheorem{assumption}[theorem]{Assumption}

\makeatletter
\def\dashint{\operatorname%
{\,\,\text{\bf--}\kern-.98em\DOTSI\intop\ilimits@\!\!}}
\makeatother

\def\sft{{\sf t}}
\def\sfx{{\sf x}}

\def\bC{\mathbb{C}}

\def\bH{\mathbb{H}}
\def\bL{\mathbb{L}}

\def\bR{\mathbb{R}}
\def\bZ{\mathbb{Z}}

\def\fH{\mathfrak{H}}

\def\cD{\mathcal{D}}

\def\cH{\mathcal{H}}

\def\cL{\mathcal{L}}
\def\cM{\mathcal{M}}

\def\cQ{\mathcal{Q}}

\def\cW{\mathcal{W}}

\begin{document}

\title[Elliptic and parabolic equations in weighted Sobolev spaces]{Elliptic and parabolic equations with measurable coefficients in weighted Sobolev spaces}

\author[H. Dong]{Hongjie Dong}
\address[H. Dong]{Division of Applied Mathematics, Brown University, 182 George Street, Providence, RI 02912, USA}

\email{Hongjie\_Dong@brown.edu}

\thanks{H. Dong was partially supported by the NSF under agreement DMS-1056737.}

\author[D. Kim]{Doyoon Kim}
\address[D. Kim]{Department of Applied Mathematics, Kyung Hee University, 1732 Deogyeong-daero, Giheung-gu, Yongin-si, Gyeonggi-do 446-701, Republic of Korea}

\email{doyoonkim@khu.ac.kr}

\thanks{D. Kim was supported by Basic Science Research Program through the National Research Foundation of Korea (NRF) funded by the Ministry of Science, ICT \& Future Planning (2011-0013960)
and by Kyung Hee University (20130747).}

\subjclass[2010]{35J25, 35K20, 35R05}

\keywords{elliptic and parabolic equations, weighted Sobolev spaces, measurable coefficients}

\begin{abstract}
We consider both divergence and non-divergence parabolic equations on a half space in weighted Sobolev spaces. All the leading coefficients are assumed to be only measurable in the time and one spatial variable except one coefficient, which is assumed to be only measurable either in the time or the spatial variable. As functions of the other variables the coefficients have small bounded mean oscillation (BMO) semi-norms. The lower-order coefficients are allowed to blow up near the boundary with a certain optimal growth condition. As a corollary, we also obtain the corresponding results for elliptic equations.
\end{abstract}

\maketitle

\section{Introduction}

In this paper we study parabolic equations in non-divergence form and divergence form:
\begin{equation*}
- u_t + a^{ij} D_{ij} u +b^i D_i u+ c u- \lambda u = f,
\end{equation*}
\begin{equation}
							\label{eq0919_2}
- u_t + D_i(a^{ij}D_j u +b^i u)+\hat b^i D_i u +  c u- \lambda u = D_i g_i +f
\end{equation}
in $(-\infty, T) \times \bR^d_+$, as well as the corresponding elliptic equations:
$$
a^{ij} D_{ij} u +b^i D_i u+ c u- \lambda u = f,
$$
$$
D_i(a^{ij}D_j u +b^i u)+\hat b^i D_i u +  c u- \lambda u = D_i g_i +f
$$
in $\bR^d_+$, where $\bR^d_+ = \{x = (x_1,x') \in \bR^d, x_1 > 0, x' \in \bR^{d-1}\}$ and $\lambda$ is a non-negative number.
We consider the equations in the weighted Sobolev spaces  $H_{p,\theta}^\gamma(\bR^d_+)$ and $\bH_{p,\theta}^\gamma((-\infty,T)\times\bR^d_+)$, which were introduced in a unified manner by N. V. Krylov  \cite{MR1708104} for all $\gamma \in \bR$.
In particular, if $\gamma$ is a non-negative integer,
$$
H_{p, \theta}^{\gamma}=H_{p, \theta}^{\gamma}(\bR^d_+)
= \{ u : x_1^{|\alpha|} D^{\alpha} u \in L_{p, \theta}(\bR^{d}_{+}) \
\forall \alpha : 0 \le |\alpha| \le \gamma \},
$$
where $L_{p, \theta}(\bR^{d}_{+})$ is an $L_p$ space with the measure
$\mu_d(dx) = x_1^{\theta - d} \, dx$.

Since the work in \cite{MR1708104}, there has been much attention to the solvability theory for equations in the weighted Sobolev spaces $H_{p,\theta}^\gamma$; see \cite{MR2111792, MR2561181, MR2990037, MR3147235}.
The necessity of such theory came from stochastic partial differential equations (SPDEs) and is well explained in \cite{MR1262972}.
For SPDEs in weighted Sobolev spaces, we refer the reader to \cite{
MR1720129, MR2073414, MR2102888, MR2519358, MR3034605}.

In this paper we extend the existing theory for equations in the weighted Sobolev spaces to a considerably more general setting. Compared to the known results in the literature, the features of our results can be summarized as follows:
\begin{itemize}
\item The leading coefficients $a^{ij}$ are in a substantially larger class of functions.

\item In the divergence case, the space of data (or {\em free terms}) is larger.

\item The lower-order coefficients are not required to approach zero as $x_1\to +\infty$.
\end{itemize}

The most significant difference from the previous results is that we allow the leading coefficients $a^{ij}$ to be merely measurable in $x_1$-direction. That is, we do not assume any regularity conditions on $a^{ij}$ as functions of $x_1$ variable.
In the parabolic case, we further allow all the leading coefficients $a^{ij}(t,x)$ to be merely measurable in $(t,x_1)$ except $a^{11}(t,x)$, which is either measurable in $t$ or in $x_1$.
As functions of the other variables, the coefficients $a^{ij}$ have small bounded mean oscillations (BMO) (see Assumptions in Section \ref{sec02}).

In the literature, the Laplace and heat equations in the weighted Sobolev spaces $H_{p,\theta}^\gamma$ were first considered in \cite{MR1708104}, when $\theta$ is in the optimal range $(d-1,d-1+p)$.
These results were extended to non-divergence type elliptic and parabolic equations with continuous coefficients in \cite{MR2111792}.
Kozlov and Nazarov \cite{MR2561181} treated parabolic equations with coefficients $a^{ij}=a^{ij}(t)$ in mixed space-time norms with the same type of weights.
Coefficients with small mean oscillations were considered in \cite{MR2519358} for SPDEs in the setting of $H_{p,\theta}^\gamma$.
Recently, in \cite{MR2990037, MR3147235} the authors treated non-divergence and divergence type equations, respectively, with coefficients having small mean oscillations. For instance, in \cite{MR3147235} the coefficients are assumed to have small mean oscillations in both the space and time variables.

The class of coefficients in this paper (called {\em partially BMO coefficients}) has been studied in \cite{MR2338417, MR2332574, MR2833589} for non-divergence type elliptic and parabolic equations, and in \cite{MR2601069, MR2764911} for divergence type equations in the usual Sobolev spaces (or Sobolev spaces without weights).
For more results on equations/systems with coefficients measurable in one spatial direction, we also refer the reader to \cite{MR2300337, MR2540989, MR2896169} in the non-divergence case and \cite{MR2800569, MR2680179, MR2835999, MR3073000} in the divergence case.
Regarding the unique solvability of equations in Sobolev spaces, it is in some sense a minimal assumption to allow the leading coefficients to be measurable in one spatial direction. In fact, the counterexamples in \cite{MR0226179, MR1612401, MR0159110} show that the unique solvability in Sobolev spaces (without weights) may fail if coefficients are merely measurable in two spatial directions.

In the divergence case, we take  larger function spaces for data on the right-hand side of the equations than those in the previous results.
For instance, in \cite{MR2519358, MR3147235} the right-hand sides of the equations under consideration have the form $D_i g_i + f$ with $g = (g_1, \ldots, g_d) \in \bL_{p,\theta}$ and $f \in M^{-1}\bH^{-1}_{p,\theta}$.  See Section \ref{sec02} for the definitions of these spaces. Thus, as explained in \cite{MR1708104} (also see \eqref{eq0102_5} in Section \ref{sec02}), $D_i g_i + f$ is indeed in $M^{-1}\bH_{p,\theta}^{-1}$.
In this paper we assume that $f$ in \eqref{eq0919_2} belongs to $\bL_{p,\theta}$ when $\lambda > 0$. Therefore, in our case the right-hand side of \eqref{eq0919_2} is in a larger space $M^{-1}\bH^{-1}_{p,\theta} + \bL_{p,\theta}$ when $\lambda > 0$.

As in \cite{MR2111792, MR2519358, MR2990037, MR3147235} we allow the lower-order coefficients to blow up at certain rates near the boundary.
On the other hand, in the previous results, those coefficients need to approach zero far away from the boundary when equations are considered in a half space.
As pointed out in \cite{MR2519358}, in some applications of PDEs or SPDEs in bounded domains this is irrelevant because far from the boundary everything is taken care of by estimates in the usual Sobolev spaces.
Nevertheless, in this paper we remove the smallness restriction on the lower-order coefficients. Instead, we assume that, as in the results for equations in the usual Sobolev spaces, away from the boundary those coefficients are only bounded. This is made possible by having more general data on the right-hand side and introducing the parameter $\lambda$ in the equations.

The overall procedure to obtain the main results is as usual by deriving a priori estimates and then using the method of continuity. In general, one first derives the a prior estimate (and the unique solvability) for relatively simple equations such as the Laplace or heat equation.
However, we cannot start with such model equations because our coefficients are merely measurable in $x_1$.
Hence a crucial step of our proof is to obtain all the necessary results for parabolic equations with simple coefficients. Here by simple coefficients we mean that they are measurable functions of only $(t,x_1)$ without any smoothness assumptions. The coefficient $a^{11}$ is either $a^{11}(t)$ or $a^{11}(x_1)$.
Then we prove certain H\"older estimates and mean oscillation estimates for equations with simple coefficients, and incorporate a perturbation argument to deduce mean oscillation estimates for equations with  partially BMO coefficients.
While establishing mean oscillation estimates and a priori estimates, we follow the idea in \cite{MR2519358} of reducing the estimate of the highest order norm of solutions to that of lower order norms, and take full advantage of the known results for equations with the same coefficients in the usual Sobolev spaces. In particular, in the non-divergence case we only estimate the mean oscillation of $Du$ instead of $D^2u$. In fact, it is not feasible to directly estimate the mean oscillation of $D^2 u$ due to the irregularity of the coefficients, even if $a^{ij}=a^{ij}(t)$ for all $i, j = 1, \ldots, d$.
Finally, we use the Fefferman-Stein theorem on sharp functions
and the Hardy-Littlewood maximal function theorem with weighted measures. One may find in \cite{MR2435520} these theorems in the forms needed for our purpose.

Following the arguments in \cite{MR2481294} (also see \cite{MR2990037, MR3147235}), as an application one can obtain the corresponding $L_p$-theory for SPDEs with the coefficients in this paper. We also note that our results can be extended to Cauchy problems with appropriate initial conditions.







The organization of the paper is as follows. In Section \ref{sec02} we state the assumptions and main results. In the subsequent sections, we deal with only parabolic equations because the results for the elliptic case follow from those for parabolic equations. We then deal with both non-divergence and divergence equations with simple coefficients in Section \ref{sec03}.
In Section \ref{sec04} we obtain mean oscillation estimates for equations with simple coefficients. Finally we prove our main results for the non-divergence case and for the divergence case in Sections \ref{sec05} and \ref{sec06}, respectively.

We finish the introduction by summarizing the notation used in this paper:
$$
M u = x_1 u,
\quad M^{-1}u = x_1^{-1}u,
\quad \bR^{d+1}_+=\bR\times \bR^d_+,
$$
$$
\bR_+ = \bR^1_+ = (0,\infty),
\quad
\bR_T = (-\infty,T),\quad \bR_T^{d+1}=\bR_T\times \bR^{d},
\quad
X=(t,x) \in \bR^{d+1}_T,
$$
$$
B_r'(x')=\{ y\in \bR^{d-1}\,|\,|y'-x'|<r\},\quad Q_r'(X')=(t-r^2,t)\times B_r'(x),
$$
$$
B_r(x)=(x_1-r, x_1+r) \times B'_r(x'),
\quad Q_r(X)=(t-r^2,t)\times B_r(x),
$$
$$
B_r^+(x)=B_r(x)\cap \bR^d_+,\quad Q_r^+(X)=(t-r^2,t)\times B_r^+(x),
$$
$$
B_r^+(x_1)=B_r^+(x_1,0), \quad Q_r^+(x_1)=Q_r^+(0,x_1,0).
$$

\section{Assumptions and main results}
							\label{sec02}

Throughout the paper, we assume that the leading coefficients $a^{ij}$ satisfy the following ellipticity condition and boundedness condition

$$
\delta|\xi|^2 \le a^{ij} \xi_i \xi_j,
\quad
|a^{ij}| \le \delta^{-1}.
$$
In the non-divergence case, without loss of generality, we assume that $a^{ij} = a^{ji}$.

We now introduce some function spaces which will be used in this paper.
When $\gamma$ is a non-negative integer, $H_{p,\theta}^\gamma$ ($=H_{p,\theta}^\gamma(\bR^d_+)$, $H_{p,\theta}^0 = L_{p,\theta}$) is introduced in the introduction. In general, if $\gamma$ is an arbitrary real number,  $H_{p,\theta}^\gamma$ is defined as follows.
Take and fix a nonnegative function $\zeta \in C_0^\infty(\bR_+)$ such that
$$
\sum_{n = -\infty}^{\infty} \zeta^p\left(e^{x_1-n}\right) \ge 1
$$
for all $x_1 \in \bR$.
For $\gamma, \theta \in \bR$, and $p \in (1,\infty)$,
let $H_{p,\theta}^\gamma$  be the set of all distributions $u$ on $\bR^d_+$ such that
$$
\| u \|_{\gamma, p,\theta}^p
:= \sum_{n = -\infty}^{\infty} e^{n \theta} \|u(e^n \cdot) \zeta(x_1)\|_{\gamma,p}^p < \infty,
$$
where $\| \cdot \|_{\gamma,p}$ is the norm of the Bessel potential space $H_p^\gamma(\bR^d)$. We recall that the operator $MD$ is bounded from $H_{p,\theta}^\gamma$ to $H_{p,\theta}^{\gamma-1}$; see \cite{MR1708104}.
For parabolic equations, we define the function spaces
$$
\bH_{p,\theta}^\gamma (S,T) = L_p \big( (S,T), H_{p,\theta}^\gamma\big), \quad
\bL_{p,\theta}(S,T) = L_p \big( (S,T), L_{p,\theta} \big),
$$
where $-\infty \le S < T \le \infty$.
Occasionally, we denote $f \in M^k L_{p,\theta}$ if $M^{-k} f \in L_{p,\theta}$, $k \in \bZ$, that is, $x_1^{-k} f \in L_{p,\theta}$, and
$f \in M^k\bL_{p,\theta}(S,T)$ if $M^{-k} f \in \bL_{p,\theta}(S,T)$.
For non-divergence type parabolic equations, we denote $u \in \fH_{p,\theta}^2(-\infty,T)$ if
$$
M^{-1} u, \,\,
Du,  \,\,
M D^2 u, \,\, M u_t \in \bL_{p,\theta}(-\infty,T).
$$
We set
$$
\|u\|_{\fH_{p,\theta}^2(-\infty,T)} = \|M^{-1} u \|_{p,\theta} + \|Du \|_{p,\theta} + \|M D^2 u \|_{p,\theta} + \|Mu_t\|_{p,\theta},
$$
where $\|\cdot\|_{p,\theta} = \|\cdot\|_{\bL_{p,\theta}(-\infty,T)}$.

For divergence type parabolic equations, we denote
$u \in \cH_{p,\theta}^{1,\lambda}(-\infty,T)$
if
$$
\sqrt \lambda u,\,\,M^{-1} u, \,\,
Du \in \bL_{p,\theta}(-\infty,T),
$$
and $u_t\in M^{-1}\bH^{-1}_{p,\theta}(-\infty,T)+
\sqrt\lambda \bL_{p,\theta}(-\infty,T)$. By Remark 5.3 in \cite{MR1708104}, for any $h \in M^{-1}\bH_{p,\theta}^{-1}(-\infty,T)$, there exists $g=(g_1,\ldots,g_d)$ satisfying
$g \in \bL_{p,\theta}(-\infty,T)$,
$D_i g_i = h$, and
\begin{equation}
							\label{eq0102_5}
\|M h\|_{\bH_{p,\theta}^{-1}(-\infty,T)} \le N \sum_i \|g_i\|_{\bL_{p,\theta}(-\infty,T)}
\le N \|M h\|_{\bH_{p,\theta}^{-1}(-\infty,T)}.
\end{equation}
Thus one can find $f, g=(g_1,\dots,g_d) \in \bL_{p,\theta}(-\infty,T)$
such that
$$
u_t = D_i g_i + \sqrt\lambda f
$$
in $\bR_T \times \bR^d_+$, or in the weak formulation
\begin{equation*}
\int_{\bR_T\times\bR^d_+} u \varphi_t \, dx \, dt
= 
\int_{\bR_T\times\bR^d_+} g_i D_i \varphi \, dx \, dt
- \sqrt \lambda \int_{\bR_T\times\bR^d_+} f \varphi \, dx \, dt
\end{equation*}
for all $\varphi \in C_0^{\infty}(\bR_T\times\bR^d_+)$. We set  \begin{align*}
&\|u\|_{\cH_{p,\theta}^{1,\lambda}(-\infty,T)}\\
&=\inf\{\sqrt\lambda\|u\|_{p,\theta}+\|M^{-1}u\|_{p,\theta}+\|Du\|_{p,\theta}
+\|g\|_{p,\theta}+\|f\|_{p,\theta}\,|\,u_t = D_i g_i + \sqrt\lambda f\}.
\end{align*}
Note that when $\lambda=0$, $\cH_{p,\theta}^{1,\lambda}(-\infty,T)=
\cH_{p,\theta}^{1}(-\infty,T)$, which is used in \cite{MR1708104}. As in \cite{MR1708104}, it is easily seen that $C_0^{\infty}((-\infty,T] \times \bR^d_+)$ is dense in $\cH_{p,\theta}^{1,\lambda}(-\infty,T)$.

We will use some results for non-divergence type parabolic equation in Sobolev spaces without weights. Recall
$$
W_p^{1,2}\left((S,T) \times \Omega \right)
= \left\{ u \, | \,u, Du, D^2 u, u_t \in L_p \left((S,T) \times \Omega \right) \right\}.
$$
Here $\Omega$ is either $\bR^d$ or $\bR^d_+$. Similarly,
for divergence type parabolic equations, we recall
$$
\cH_p^1\left((S,T) \times \Omega\right)
= \left\{ u \, | \, u, Du \in L_p\left((S,T) \times \Omega\right), u_t \in \bH_p^{-1}\left((S,T) \times \Omega\right)\right\},
$$
where
$$
\bH_p^{-1}\left((S,T) \times \Omega\right)
= \left\{v \, | \, v = D_i g_i + h, \, g_i, h \in L_p\left((S,T) \times \Omega\right) \right\}.
$$

In order to state another assumption on the coefficients $a^{ij}$,
we introduce the following notation.
Set $B_r' = B_r'(0)$ and $|B'_r|$ to be the volume of $B'_r$.
We also set
$$
\mu_{d}(dx)=x_1^{\theta-d} \, dx,
\quad
\mu_d(\Omega) = \int_{\Omega} \mu_{d}(dx),
$$
$$
\mu_{d+1}(dx\,dt)=x_1^{\theta-d} \, dx \, dt,\quad \mu_{d+1}(\cD) = \int_{\cD} \mu_{d+1}(dx\,dt),
$$
for $\Omega \subset \bR^d_+$ and $\cD \subset \bR^{d+1}_+$.
Throughout the paper, unless specified otherwise, $\mu$ means $\mu_{d+1}$, a measure on $\bR^{d+1}_+$.

For a function $g$ on $\bR^{d+1}_+$, denote
$$
\left[ g(t,\cdot)\right]_{B_r(x)}=
\dashint_{B_r(x)}
\left| g(t, y) -
\dashint_{B_r(x)}
g(t,z) \, \mu_d(dz) \right|\, \mu_d(dy),
$$
$$
\left[ g(t,x_1,\cdot)\right]_{B'_r(x')}=
\dashint_{B'_r(x')}\left| g(t, x_1, y') - \dashint_{B'_r(x')} g(t,x_1, z')\, dz' \right| \, dy',
$$
$$
\left[ g(\cdot,x_1,\cdot)\right]_{Q'_r(t,x')}=
\dashint_{Q'_r(t,x')}\left| g(t, x_1, y') -
\dashint_{Q'_r(t,x')}g(s,x_1,z') \,ds\,dz' \right| \, dt\,dy'.
$$
Then we define the mean oscillation of $g$ in $Q_r(s,y)$ with respect to $x$ as
$$
\text{osc}_{\sfx} \left(g,Q_r(s,y)\right)
= \dashint_{s-r^2}^{\,\,s}
\left[ g(\tau, \cdot)\right]_{B_r(y)} \, d\tau,
$$
and, for $\rho\in (1/2,1)$, denote
$$
g^{\sfx,\#}_\rho =\sup_{(s,y) \in \bR^{d+1}_+}\sup_{r\in (0, \rho y_1]}\text{osc}_\sfx\left(g,Q_r(s,y)\right).
$$
We also define the mean oscillation of $g$ in $Q_r(s,y)$ with respect to $(t,x')$ as
$$
\text{osc}_{(\sft,\sfx')} \left(g,Q_r(s,y)\right)
= \dashint_{y_1-r}^{\,\,y_1+r} \left[ g(\cdot,z_1, \cdot)\right]_{Q_r'(s,y')} \mu_1(dz_1),
$$
and denote
$$
g^{(\sft, \sfx'),\#}_\rho =\sup_{(s,y) \in \bR^{d+1}_+}\sup_{r\in (0, \rho y_1]}\text{osc}_{(\sft,\sfx')}\left(g,Q_r(s,y)\right).
$$
Furthermore, we define the mean oscillation of $g$ in $Q_r(s,y)$ with respect to $x'$ as
$$
\text{osc}_{\sfx'}
\left(g,Q_r(s,y)\right)
=  \dashint_{s-r^2}^{\,\,s}
\dashint_{y_1-r}^{\,\,y_1+r}
\left[ g(\tau,z_1,\cdot)\right]_{B'_r(y')} \mu_1(d z_1) \, d\tau,
$$
and denote
$$
g^{\sfx', \#}_\rho =\sup_{(s,y) \in \bR^{d+1}_+}\sup_{r\in (0, \rho y_1]}\text{osc}_{\sfx'}\left(g,Q_r(s,y)\right).
$$
In the case when $g$ is independent of $t$,
i.e., if $g$ is a function of $x \in \bR^d_+$, we set
$$
\text{osc}_{\sfx'} \left(g,B_r(y)\right)
= \dashint_{y_1-r}^{\,\,y_1+r} \left[ g(z_1,\cdot) \right]_{B'_r(y')} \mu_1(dz_1),
$$
$$
g^{\sfx',\#}_\rho=\sup_{y \in \bR^{d}_+}\sup_{r\in (0, \rho y_1]} \text{osc}_{\sfx'}
\left(g,B_r(y)\right).
$$

Using the above notation with $a^{ij}$ in place of $g$
we state the following regularity assumptions on $a^{ij}$,
where the parameters $\rho\in (1/2,1)$ sufficiently close to 1 and $\varepsilon>0$ sufficiently small will be specified later.

\begin{AssumpA}
We have
$$
(a^{11})_{\rho}^{\sfx,\#}
+\sum_{ij>1}(a^{ij})_{\rho}^{\sfx',\#} \le \varepsilon.
$$
\end{AssumpA}

\begin{AssumpA'}
We have
$$
(a^{11})_{\rho}^{(\sft,\sfx'),\#}
+\sum_{ij>1}(a^{ij})_{\rho}^{\sfx',\#} \le \varepsilon.
$$
\end{AssumpA'}	

Now we state the main results of the paper.

\begin{theorem}
                            \label{thm1}
Let $T \in (-\infty, \infty]$, $\lambda \ge 0$, $1< p < \infty$, and $\theta\in (d-1,d-1+p)$. Then there exist positive constants $\rho\in (1/2,1)$, $\varepsilon$, and $\varepsilon_1$ depending only on $d$, $\delta$, $p$, and $\theta$ such that under Assumption A ($\rho,\varepsilon$) or Assumption A$'$ ($\rho,\varepsilon$), and the growth condition
\begin{equation}
                                    \label{eq9.34}
|x_1 b^i|+|x_1^2 c|\le \varepsilon_1,
\end{equation}
the following assertions hold.

(i) Suppose that $u \in \fH_{p,\theta}^2(-\infty,T)$ satisfies
\begin{equation}
							\label{eq9.22}
- u_t + a^{ij} D_{ij} u +b^i D_i u+ c u- \lambda u = f
\end{equation}
in $\bR_T \times \bR^d_+$,
where $f \in M^{-1}\bL_{p,\theta}(-\infty,T)$.
Then
\begin{equation}
							\label{eq9.23}
\lambda \| M u \|_{p,\theta}
+ \| M^{-1} u \|_{p,\theta} + \| D u \|_{p,\theta}
+ \| M D^2 u \|_{p,\theta} + \| M u_t \|_{p,\theta}
\le N \|Mf\|_{p,\theta},
\end{equation}
where $\| \cdot \|_{p,\theta} = \| \cdot \|_{\bL_{p,\theta}(-\infty,T)}$ and $N = N(d, \delta, \theta, p)$.

(ii) For any $f \in M^{-1}\bL_{p,\theta}(-\infty,T)$, there is a unique solution $u \in \fH_{p,\theta}^2(-\infty,T)$ to the equation \eqref{eq9.22}.

(iii) If the condition \eqref{eq9.34} is satisfied only for $x_1 \in (0,\sigma]$ for some $\sigma \in (0,\infty)$, and $|b^i|, |c| \le K$ for $x_1 \in (\sigma,\infty)$,
then there exists a constant $\lambda_0 \ge 0$ depending only on $d$, $\delta$, $p$, $\theta$, and $K$ such that the above two assertions hold true whenever $\lambda \ge \lambda_0$.
\end{theorem}

\begin{theorem}
                            \label{thm2}
Let $T \in (-\infty, \infty]$, $\lambda \ge 0$, $1< p < \infty$, and $\theta\in (d-1, d-1+p)$. Then there exist positive constants $\rho\in (1/2,1)$, $\varepsilon$, and $\varepsilon_1$ depending only on $d$, $\delta$, $p$, and $\theta$ such that under Assumption A ($\rho,\varepsilon$) or Assumption A$'$ ($\rho,\varepsilon$), and the growth condition
\begin{equation}
                                    \label{eq9.34b}
|x_1 b^i|+|x_1\hat b^i|+|x_1^2 c|\le \varepsilon_1,
\end{equation}
the following assertions hold.

(i) Suppose that $u \in \cH_{p,\theta}^{1,\lambda}(-\infty,T)$ satisfies
\begin{equation}
							\label{eq9.22b}
- u_t + D_i(a^{ij}D_j u +b^i u)+\hat b^i D_i u +  c u- \lambda u = D_i g_i +f
\end{equation}
in $\bR_T \times \bR^d_+$,
where $g =(g_1,\ldots,g_d), f \in \bL_{p,\theta}(-\infty,T)$ and $f \equiv 0$ if $\lambda = 0$.
Then
\begin{equation}
							\label{eq9.23b}
\sqrt{\lambda} \| u \|_{p,\theta}
+ \| M^{-1} u \|_{p,\theta} + \| D u \|_{p,\theta}
\le N \|g\|_{p,\theta}+N\lambda^{-1/2} \|f\|_{p,\theta},
\end{equation}
where $\| \cdot \|_{p,\theta} = \| \cdot \|_{\bL_{p,\theta}(-\infty,T)}$ and $N = N(d, \delta, \theta, p)$.

(ii) For $g, f \in \bL_{p,\theta}(-\infty,T)$ such that $f \equiv 0$ if $\lambda = 0$, there exists a unique solution $u \in \cH_{p,\theta}^{1,\lambda}(-\infty,T)$ to the equation \eqref{eq9.22b}.

(iii) If the condition \eqref{eq9.34b} is satisfied only for $x_1 \in (0,\sigma]$ for some $\sigma \in (0,\infty)$, and $|b^i|, |\hat b^i|, |c| \le K$ for $x_1 \in (\sigma,\infty)$,
then there exists a constant $\lambda_0 \ge 0$ depending only on $d$, $\delta$, $p$, $\theta$, and $K$ such that the above two assertions hold true whenever $\lambda \ge \lambda_0$.
\end{theorem}

For divergence type elliptic equations, we denote
$u \in \cW_{p,\theta}^{1,\lambda} \equiv \cW_{p,\theta}^{1,\lambda}(\bR^d_+)$
if $
\sqrt \lambda u,M^{-1} u,Du \in L_{p,\theta}$
with the norm
$$
\|u\|_{\cW_{p,\theta}^{1,\lambda}}
=\sqrt\lambda\|u\|_{p,\theta}+\|M^{-1}u\|_{p,\theta}+\|Du\|_{p,\theta}.
$$

We impose the following regularity assumption on $a^{ij}$ for elliptic equations.
\begin{AssumpA''}
We have
$$
\sum_{i,j=1}(a^{ij})_{\rho}^{\sfx',\#} \le \varepsilon.
$$
\end{AssumpA''}	

By adapting, for example, the proof of Theorem 2.6 in \cite{MR2304157} to the results above for parabolic equations, we obtain the following theorems for elliptic equations.

\begin{theorem}
Let $\lambda \ge 0$, $1< p < \infty$, and $\theta\in (d-1,d-1+p)$. Then there exist positive constants $\rho\in (1/2,1)$, $\varepsilon$, and $\varepsilon_1$ depending only on $d$, $\delta$, $p$, and $\theta$ such that under Assumption A$''$ ($\rho,\varepsilon$) and the growth condition
\begin{equation}
                                   \label{eq9.341}
|x_1 b^i|+|x_1^2 c|\le \varepsilon_1,
\end{equation}
the following assertions hold.

(i) Suppose that $u \in M H_{p,\theta}^2$ satisfies
\begin{equation}
							\label{eq9.221}
a^{ij} D_{ij} u +b^i D_i u+ c u- \lambda u = f
\end{equation}
in $\bR^d_+$,
where $f \in M^{-1}L_{p,\theta}$.
Then
\begin{equation*}
\lambda \| M u \|_{p,\theta}
+ \| M^{-1} u \|_{p,\theta} + \| D u \|_{p,\theta}
+ \| M D^2 u \|_{p,\theta}
\le N \|Mf\|_{p,\theta},
\end{equation*}
where $\| \cdot \|_{p,\theta} = \| \cdot \|_{L_{p,\theta}}$ and $N = N(d, \delta, \theta, p)$.

(ii) For any $f \in M^{-1}L_{p,\theta}$, there exists a unique solution $u \in M H_{p,\theta}^2$ to the equation \eqref{eq9.221}.

(iii) If the condition \eqref{eq9.341} is satisfied only for $x_1 \in (0,\sigma]$ for some $\sigma \in (0,\infty)$, and $|b^i|, |c| \le K$ for $x_1 \in (\sigma,\infty)$,
then there exists a constant $\lambda_0 \ge 0$ depending only on $d$, $\delta$, $p$, $\theta$, and $K$ such that the above two assertions hold true whenever $\lambda \ge \lambda_0$.
\end{theorem}

\begin{theorem}
Let $\lambda \ge 0$, $1< p < \infty$, and $\theta\in (d-1,d-1+p)$. Then there exist positive constants $\rho\in (1/2,1)$, $\varepsilon$, and $\varepsilon_1$ depending only on $d$, $\delta$, $p$, and $\theta$ such that under Assumption A$''$ ($\rho,\varepsilon$) and the growth condition
\begin{equation}
                                    \label{eq9.34b1}
|x_1 b^i|+|x_1\hat b^i|+|x_1^2 c|\le \varepsilon_1,
\end{equation}
the following assertions hold.

(i) Suppose that $u \in \cW_{p,\theta}^{1,\lambda}$ satisfies
\begin{equation}
							\label{eq9.22b1}
D_i(a^{ij}D_j u +b^i u)+\hat b^i D_i u +  c u- \lambda u = D_i g_i + f
\end{equation}
in $\bR^d_+$,
where $g=(g_1,\ldots,g_d), f \in L_{p,\theta}$ and $f \equiv 0$ if $\lambda = 0$.
Then
\begin{equation*}
\sqrt{\lambda} \| u \|_{p,\theta}
+ \| M^{-1} u \|_{p,\theta} + \| D u \|_{p,\theta}
\le N \|g\|_{p,\theta}+N\lambda^{-1/2} \|f\|_{p,\theta},
\end{equation*}
where $\| \cdot \|_{p,\theta} = \| \cdot \|_{L_{p,\theta}}$ and $N = N(d, \delta, \theta, p)$.

(ii) For $g, f \in L_{p,\theta}$ such that $f \equiv 0$ if $\lambda = 0$, there exists a unique solution $u \in \cW_{p,\theta}^{1,\lambda}$ to the equation \eqref{eq9.22b1}.

(iii) If the condition \eqref{eq9.34b1} is satisfied only for $x_1 \in (0,\sigma]$ for some $\sigma \in (0,\infty)$, and $|b^i|, |\hat b^i|, |c| \le K$ for $x_1 \in (\sigma,\infty)$,
then there exists a constant $\lambda_0 \ge 0$ depending only on $d$, $\delta$, $p$, $\theta$, and $K$ such that the above two assertions hold true whenever $\lambda \ge \lambda_0$.
\end{theorem}

\section{$L_p$-estimates for equations with simple coefficients}
							\label{sec03}

In this section we consider parabolic equations with simple coefficients.
Throughout the section the following assumption is enforced.

\begin{assumption}
                    \label{assum3.1}
$$
a^{ij}=a^{ij}(t,x_1)
\quad
\text{for}
\quad
(i,j) \ne (1,1).
$$
$$
a^{11} = a^{11}(t)
\quad
\text{or}
\quad
a^{11} = a^{11}(x_1).
$$
\end{assumption}

\begin{lemma}
							\label{lem0920_1}
Let $1<p<\infty$, $1-p < c < 1$, and $v \in C_0^{\infty}(\bR^d_+)$.
We have
$$
\int_{\bR^d_+} |v|^p \, y_1^{c-2} \, dy
\le \frac{p^2}{(1-c)^2} \int_{\bR^d_+} |v|^{p-2} \left(D_1 v\right)^2 \, y_1^c \, dy.
$$
\end{lemma}

\begin{proof}
This is Hardy's inequality.
Indeed, one can find, for example, in \cite{MR802206}
$$
\int_0^\infty |u(r)|^2 r^{c-2} \, dr \le \frac{4}{(1-c)^2}
\int_0^\infty |u'(r)|^2 r^c \, dr
$$
if $c < 1$ and $u(r)$ is a sufficiently smooth function defined in $[0,\infty)$ satisfying $u(0) = 0$.
Then using this inequality with $u(r) = |v(r,x')|^{p/2}$ and integrating both sides with respect to $x' \in \bR^{d-1}$, we get the desired inequality.
One can find the same inequality in the proof of Lemma 6.1 in \cite{MR1708104}.
\end{proof}

\subsection{Non-divergence type equations}

We start with estimating the weighted norm of $M^{-1}u$ by generalizing Corollary 6.2 of \cite{MR1708104}, where the result was proved for equations with constant coefficients.

\begin{proposition}
							\label{prop0925_1}
Let $T \in (-\infty, \infty]$, $\lambda \ge 0$, $1< p < \infty$, $1-p < c < 1$, and $u \in C_0^{\infty}((-\infty,T] \times \bR^d_+)$ satisfy
\begin{equation}
							\label{eq0919_1}
- u_t + a^{ij} D_{ij} u - \lambda u = f
\end{equation}
in $\bR_T \times \bR^d_+$,
where
$$
\int_{\bR_T \times \bR^d_+} |f|^p \, x_1^{c-2+2p} \, dx \, dt < \infty.
$$
Then
\begin{equation}
							\label{eq0925_1}
\int_{\bR_T \times \bR^d_+} |u|^p \, x_1^{c-2} \, dx \, dt
\le N \int_{\bR_T \times \bR^d_+} |f|^p \, x_1^{c-2+2p} \, dx \, dt,
\end{equation}
where $N = N(d, \delta, c, p)$.
\end{proposition}

\begin{proof}
{\bf Case 1}: $a^{11} = a^{11}(t)$.
Multiply both sides of \eqref{eq0919_1} by $- |u|^{p-2} u \, x_1^c$ and integrate over $(-\infty, T)\times \bR^d_+$.
Then we have
\begin{align}
							\label{eq0920_1}
&\int_{\bR_T\times \bR^d_+} u_t |u|^{p-2} u \, x_1^c \, dx \, dt
- \int_{\bR_T \times \bR^d_+} a^{11}(t) |u|^{p-2} u D_{11} u \, x_1^c \, dx \, dt
\nonumber\\
&\,\,- \sum_{(i,j) \ne (1,1)} \int_{\bR_T \times \bR^d_+} a^{ij} |u|^{p-2} u D_{ij} u \, x_1^c \, dx \, dt + \lambda \int_{\bR_T \times \bR^d_+} |u|^p \, x_1^c \, dx \, dt
\nonumber\\
&= - \int_{\bR_T \times \bR^d_+} f |u|^{p-2} u \, x_1^c \, dx \, dt.
\end{align}
Note that by integration by parts
\begin{align*}
\int_{\bR_T\times \bR^d_+} u_t |u|^{p-2} u \, x_1^c \, dx \, dt
&= \int_{\bR_T\times \bR^d_+} \frac{1}{p} D_t \left( |u|^p \right) \, x_1^c \, dx \, dt
\\
&= \int_{\bR^d_+} \frac{1}{p} |u|^p(T,x) \, x_1^c \, dx
\end{align*}
and
\begin{align*}
&-\int_{\bR_T \times \bR^d_+} a^{11}(t) |u|^{p-2} u D_{11} u \, x_1^c \, dx \, dt
\\
&=(p-1) \int_{\bR_T \times \bR^d_+} a^{11}(t) D_1 u D_1u |u|^{p-2} \, x_1^c \, dx \, dt
\\
&\quad+ c \int_{\bR_T \times \bR^d_+} a^{11}(t) |u|^{p-2} u D_1u \, x_1^{c-1} \, dx \, dt,
\end{align*}
where the last term is equal to
$$
\frac{c(1-c)}{p}\int_{\bR_T \times \bR^d_+} a^{11}(t) |u|^p \, x_1^{c-2} \, dx \, dt
$$
due to $\frac{1}{p} D_1\left(|u|^p \right) = |u|^{p-2} u D_1u$ and integration by parts again in $x_1$.
For $(i,j) \ne (1,1)$,
$$
- \int_{\bR_T \times \bR^d_+} a^{ij} |u|^{p-2} u D_{ij} u \, x_1^c \, dx \, dt
= (p-1) \int_{\bR_T \times \bR^d_+} a^{ij} D_i u D_j u |u|^{p-2} \, x_1^c \, dx \, dt.
$$
Thus from \eqref{eq0920_1} combined with the above calculations we have
\begin{align}
							\label{eq0920_2}
&\frac{1}{p} \int_{\bR_T \times \bR^d_+} |u|^p(T,x) \, x_1^c \, dx \, dt
+ \int_{\bR_T} I(t) \, dt
+ \lambda \int_{\bR_T \times \bR^d_+} |u|^p \, x_1^c \, dx \, dt
\nonumber\\
&= - \int_{\bR_T \times \bR^d_+} f |u|^{p-2} u \, x_1^c \, dx \, dt,
\end{align}
where
$$
I(t) = (p-1) \int_{\bR^d_+} a^{ij} D_i u D_j u |u|^{p-2} \, x_1^c \, dx
+ \frac{c(1-c)}{p}\int_{\bR^d_+} a^{11}(t) |u|^p \, x_1^{c-2} \, dx.
$$

Now, for each $t \in (-\infty, T]$, we consider a change of variables $y=y(t,x)$, where
$$
y_1 = x_1,
\quad
y_i = - \int_0^{x_1} \frac{a^{i1}(t,r)}{a^{11}(t)} \, dr + x_i,
\quad
i = 2, \ldots, d.
$$
Then
$y(0)=0$,
$\partial y_i/\partial x_i=1$,
$$
\partial y_i/\partial x_1 =-\frac{a^{i1}(t,x_1)}{a^{11}(t)},
\quad
i=2, \ldots, d,
$$
and
$$
\partial y_i/\partial x_j =0,
\quad
i \ne j, \,\, j = 2, \ldots, d.
$$
This is a one to one Lipschitz map from $\bR^d_+$ to $\bR^d_+$ and its Jacobian is equal to $1$.
Set
$v(t,y) = u(t,x)$.
Then, for each $t \in (-\infty, T]$,
$$
I(t)
= (p-1) \int_{\bR^d_+} \tilde{a}^{kl} D_k v D_l v |v|^{p-2} \, y_1^c \, dy
+ \frac{c(1-c)}{p} \int_{\bR^d_+} \tilde{a}^{11}(t) |v|^p \, y_1^{c-2} \, dy,
$$
where
$$
\tilde{a}^{kl} = \sum_{i,j=1}^d a^{ij}\frac{\partial y_k}{\partial x_i} \frac{\partial y_l}{\partial x_j}.
$$
By the definition of $y=y(t,x)$ we observe that
$$
\tilde{a}^{11} = a^{11}(t),
\quad
\tilde{a}^{1l} = \tilde{a}^{k1} = 0,
\quad
k, l = 2, \ldots, d.
$$
Hence
\begin{align*}
I(t)& = (p-1) \int_{\bR^d_+} a^{11}(t) (D_1v)^2 |v|^{p-2} \, y_1^c \, dy
+ \frac{c(1-c)}{p} \int_{\bR^d_+} a^{11}(t) |v|^p \, y_1^{c-2} \, dy\\
&\quad + (p-1) \sum_{k,l=2}^d \int_{\bR^d_+} \tilde{a}^{kl} D_k v D_l v |v|^{p-2} \, y_1^c \, dy\\
&\ge \frac{(1-c)(p-1+c)}{p^2} \int_{\bR^d_+} a^{11}(t) |v|^p \, y_1^{c-2} \, dy\\
&\ge \delta \frac{(1-c)(p-1+c)}{p^2} \int_{\bR^d_+}
|u|^p \, x_1^{c-2} \, dx,
\end{align*}
where we used Lemma \ref{lem0920_1} and the ellipticity condition of $a^{ij}$.
In particular, the latter implies  $a^{11}(t) \ge \delta$ and the ellipticity condition of $\tilde{a}^{kl}$, $k,l \ge 2$, so we see that
$$
\sum_{k,l=2}^d \int_{\bR^d_+} \tilde{a}^{kl} D_k v D_l v |v|^{p-2} \, y_1^c \, dy
\ge \delta \int_{\bR^d_+} |Dv|^2 |v|^{p-2} y_1^c \, dy \ge 0.
$$
Then using the above estimate of $I(t)$ and the non-negativity of the first term in \eqref{eq0920_2}, we obtain
\begin{align*}
&\delta_1 \int_{\bR_T \times \bR^d_+} |u|^p \, x_1^{c-2} \, dx \, dt
+ \lambda \int_{\bR_T \times \bR^d_+} |u|^p \, x_1^c \, dx \, dt
\le - \int_{\bR_T \times \bR^d_+} f |u|^{p-2} u \, x_1^c \, dx \, dt\\
&\le \left(\int_{\bR_T \times \bR^d_+} |f|^p \, x_1^{c-2+2p} \, dx \, dt \right)^{1/p}
\left(\int_{\bR_T \times \bR^d_+} |u|^p \, x_1^{c-2} \, dx \, dt \right)^{(p-1)/p},
\end{align*}
where
$$
\delta_1 = \delta \frac{(1-c)(p-1+c)}{p^2} > 0.
$$
This clearly shows \eqref{eq0925_1}.

{\bf Case 2}: $a^{11} = a^{11}(x_1)$.
In this case, we multiply both sides of \eqref{eq0919_1} by $|u|^{p-2} u \, x_1^c /a^{11}$ and proceed as above.
One noteworthy step is
$$
\int_{\bR_T\times \bR^d_+} u_t |u|^{p-2} u \, \frac{x_1^c}{a^{11}(x_1)} \, dx \, dt
= \int_{\bR^d_+} \frac{x_1^c}{a^{11}(x_1)} \int_{\bR_T}  \frac{1}{p} D_t \left( (u^2)^{\frac{p}{2}} \right) \, dt \, dx
$$
$$
= \int_{\bR^d_+} \frac{1}{p} |u|^p(T,x) \, \frac{x_1^c}{a^{11}(x_1)} \, dx \ge 0,
$$
where we made use of the fact that $a^{11}$ is independent of time.
\end{proof}

Once we have the estimate \eqref{eq0925_1} for solutions of \eqref{eq0919_1}, using the $L_p$-estimates, developed, for example, in \cite{MR2338417, MR2300337, MR2332574, MR2896169, MR2833589},  for equations with measurable coefficients in Sobolev spaces without weights, we obtain the following theorem.

\begin{theorem}[Non-divergence case]
							\label{thm1011}
Let $T \in (-\infty, \infty]$, $\lambda \ge 0$, $1< p < \infty$, $\theta\in (d-1,d-1+p)$, and $u \in \fH_{p,\theta}^2(-\infty,T)$ satisfy
\begin{equation}
							\label{eq1002_1}
- u_t + a^{ij} D_{ij} u - \lambda u = f
\end{equation}
in $\bR_T \times \bR^d_+$,
where $Mf \in \bL_p(-\infty,T)$.
Then
\begin{equation}
							\label{eq1002_2}
\lambda \| M u \|_{p,\theta}
+ \| M^{-1} u \|_{p,\theta} + \| D u \|_{p,\theta}
+ \| M D^2 u \|_{p,\theta} + \| M u_t \|_{p,\theta}
\le N \|Mf\|_{p,\theta},
\end{equation}
where $\| \cdot \|_{p,\theta} = \| \cdot \|_{\bL_{p,\theta}(-\infty,T)}$ and $N = N(d, \delta, \theta, p)$.

Moreover, for any $f \in M^{-1}\bL_p(-\infty,T)$, there exists a unique solution $u \in \fH_{p,\theta}^2(-\infty,T)$ to the equation \eqref{eq1002_1}.
\end{theorem}

\begin{proof}
The proof is similar to that of Lemma 2.2 in \cite{MR1690093}.
We first prove the estimate \eqref{eq1002_2}.
Note that because
$$
\lambda u = -u_t + a^{ij}D_{ij} u  - f \in M^{-1} \bL_p(-\infty,T),
$$
we have $\lambda M u \in \bL_p(-\infty,T)$. Then one can find $u_n \in C_0^\infty\left((-\infty,T] \times \bR^d_+\right)$ such that (see Theorem 1.19 and Remark 5.5 in \cite{MR1708104})
$$
\|u_n - u\|_{\fH_{p,\theta}^2(-\infty,T)} \to 0,
\quad
\lambda \|Mu_n - Mu\|_{\bL_{p,\theta}(-\infty,T)} \to 0
$$
as $n \to \infty$. Hence it suffices to prove \eqref{eq1002_2} for $u \in C_0^\infty\left((-\infty,T] \times \bR^d_+\right)$.
Take a function $\zeta=\zeta(x_1) \in C_0^\infty(\bR^+)$ such that
$$
\int_0^\infty r^{-1-\theta+d-p}|\zeta(r)|^p \, dr = 1.
$$
Then $u \zeta$ satisfies
$$
- \left(u \zeta\right)_t
+ a^{ij} D_{ij} \left(u \zeta \right) - \lambda u \zeta
= f \zeta + 2 a^{i1} \zeta'D_i u + a^{11}\zeta''u.
$$
Note that $u \zeta \in W_p^{1,2}\left((-\infty,T) \times \bR^d\right)$
and $a^{ij}$ satisfy the assumptions of Theorems 3.1 and 4.1 in \cite{MR2833589}. Thus we have
\begin{equation}
							\label{eq1004_1}
\lambda \|u\zeta\|_p + \|D^2\left(u\zeta\right)\|_p
+ \|\left(u\zeta\right)_t\|_p
\le N \left( \|f \zeta\|_p + \|\zeta' Du \|_p + \|\zeta'' u\|_p\right),
\end{equation}
where $\| \cdot \|_p = \| \cdot \|_{L_p\left((-\infty,T) \times \bR^d\right)}$ and $N = N(d,\delta,p)$.
We write
$$
 \zeta D_{ij} u= D_{ij}\left( u \zeta \right) - 2 D_i u D_j \zeta - u D_{ij} \zeta,
$$
from which and \eqref{eq1004_1}, we get
$$
\lambda^p \|u \zeta\|^p_p + \|\zeta D_{ij} u\|_p^p
\le N(d,\delta,p) \left(\|f \zeta\|^p_p + \| \zeta' D u\|^p_p  + \|\zeta'' u \|^p_p\right).
$$
Now we substitute here $\zeta_r(x_1) = \zeta (r x_1)$ in place of $\zeta$, where $r > 0$ is a parameter, multiply both sides of the inequality by $r^{-1-\theta+d-p}$, and integrate with respect to $r$ over $(0,\infty)$.
Then
\begin{equation}
							\label{eq1004_4}
\lambda^p \|M u\|_{p,\theta}^p
+ \| M D^2 u \|_{p,\theta}^p
\le N \left( \|M f\|_{p,\theta}^p + \| Du \|_{p,\theta}^p + \|M^{-1} u \|_{p,\theta}^p \right),
\end{equation}
where $N=N(d,\delta,\theta,p)$.
Note that by Proposition \ref{prop0925_1} with $c = 2 + \theta - d - p$ we have
\begin{equation}
							\label{eq1004_5}
\|M^{-1} u \|^p_{p,\theta}
\le N(d, \delta, \theta, p) \|M f\|_{p,\theta}^p.
\end{equation}
To estimate $\|Du\|_{p,\theta}$, by setting $\eta_r(x_1) = (r x_1)^{-1}\zeta(r x_1)$, we observe that
\begin{align}
							\label{eq1004_3}
\| Du \|_{p,\theta}^p
&= \int_0^\infty \| D (u \eta_r) - u D ( \eta_r ) \|^p_p \, r ^{d - \theta -1} \, dr
\nonumber \\
&\le 2^p \int_0^\infty \| D (u \eta_r) \|_p^p \, r^{d - \theta -1} \, dr
+ N(d,\theta,p) \| M^{-1} u \|_{p,\theta}.
\end{align}
By the interpolation inequality
\begin{equation*}
\| D(u \eta_r) \|_p^p
\le \varepsilon \| D^2(u \eta_r) \|^p_p + N\varepsilon^{-1} \| u \eta_r \|^p_p,
\end{equation*}
which combined with
$$
\|D^2 (u \eta_r) \|_p
\le N \left( \| \eta_r D^2 u \|_p + \|\eta_r' D u \|_p
+ \| \eta_r'' u \|_p \right)
$$
gives
$$
\|D (u \eta_r)\|_p^p
\le \varepsilon \| \eta_r D^2 u \|^p_p + \varepsilon \|\eta_r' D u \|^p_p
+ \varepsilon \| \eta_r'' u \|^p_p
+ N \varepsilon^{-1} \|u \eta_r\|_p^p.
$$
By plugging this with $\varepsilon r^{-p}$ in place of $\varepsilon$ into \eqref{eq1004_3}, we get
$$
\|D u\|_{p,\theta}^p
\le \varepsilon \| M D^2 u \|_{p,\theta}^p
+ \varepsilon
\| D u \|_{p, \theta}^p
+ \left(\varepsilon + N\varepsilon^{-1} \right) \| M^{-1} u \|_{p,\theta}^p,
$$
where $N = N(d, \theta, p)$.
From this with a sufficiently small $\varepsilon < 1/2$,
and the inequalities \eqref{eq1004_4} and \eqref{eq1004_5}, we obtain
$$
\lambda \| Mu \|_{p,\theta}
+ \| M^{-1} u \|_{p,\theta} + \| D u \|_{p,\theta} + \| M D^2 u \|_{p,\theta}
\le N \| M f \|_{p,\theta},
$$
where $N = N(d,\delta,\theta,p)$.
Finally to get \eqref{eq1002_2}, we use the equation \eqref{eq1002_1}.

Now that we have an a priori estimate, thanks to the method of continuity, to prove the second assertion of the theorem for unique solvability, we only need to prove the solvability of $-u_t + \Delta u - \lambda u = f$.
The case when $\lambda = 0$ follows from \cite[Lemma 5.7]{MR1708104}.
For $\lambda > 0$,
due to Remark 5.5 in \cite{MR1708104} and the a priori estimate \eqref{eq1002_2}, we assume that $f \in C_0^\infty\left((-\infty,T] \times \bR^d_+\right)$. In this case there exists a solution $u$ with $u(t,0,x') = 0$, which is infinitely differentiable and belongs to $W_p^{1,2}\left((-\infty,T) \times \bR^d_+\right)$.
Then we obtain the solvability if we show that $u \in \fH_{p,\theta}^2(-\infty,T)$.
To prove this, we follow the lines described above for the proof of \eqref{eq1002_2} once we check  that
$M^{-1} u \in \bL_{p,\theta}(-\infty,T)$, which follows easily
from Hardy's inequality (for instance, see Theorem 5.2 in \cite{MR802206}) and the fact that $u \in W_p^{1,2}\left((-\infty,T) \times \bR^d_+\right)$, so in particular,
\begin{equation}
							\label{eq1221_4}
\begin{aligned}
\int_{\bR_T \times \bR^d_+}
1_{x_1 \in (1,\infty)} |u|^p \, x_1^{\theta - d - p} \, dx \, dt &< \infty,
\\
\int_{\bR_T \times \bR^d_+}
1_{x_1 \in (0,1)} |D^2 u|^p \, x_1^{\theta - d + p} \, dx \, dt &< \infty.
\end{aligned}
\end{equation}
The theorem is proved.
\end{proof}

\begin{remark}
							\label{rem12_9}
When $\lambda > 0$ and  $a^{ij}$ satisfy Assumption \ref{assum3.1},
the above proof shows that, for $f \in L_p\left( (-\infty, T) \times \bR^d_+\right)$, a solution  $u \in W_p^{1,2}\left((-\infty,T) \times \bR^d_+\right)$ to the equation
$$
-u_t + a^{ij} D_{ij} u - \lambda u = f
$$
in $(-\infty,T) \times \bR^d_+$ with the boundary condition $u(t,0,x') = 0$
also belongs to $\fH_{p,\theta}^2(-\infty,T)$
if  $M f \in \bL_{p,\theta}(-\infty,T)$.
Indeed, we can repeat the same argument (see the inequalities in \eqref{eq1221_4}) used in the above proof when we show that the solution $u \in W_p^{1,2}\left((-\infty,T) \times \bR^d_+\right)$ to the equation $-u_t + \Delta u - \lambda u = f$ with $u(t,0,x') = 0$ is in $\fH_{p,\theta}^2(-\infty,T)$.
\end{remark}

\subsection{Divergence type equations}

We begin with a result analogous to Proposition \ref{prop0925_1} for divergence form equations.

\begin{proposition}
							\label{prop1009}
Let $T \in (-\infty, \infty]$, $\lambda \ge 0$, $2 \le p < \infty$, $1-p < c < 1$, and $u \in C_0^{\infty}((-\infty,T] \times \bR^d_+)$ satisfy
\begin{equation}
							\label{eq0925_2}
- u_t + D_i \left( a^{ij} D_j u \right) - \lambda u = D_i g_i + f
\end{equation}
in $\bR_T \times \bR^d_+$,
where $g = (g_1,\ldots,g_d)$, $f \equiv 0$ if $\lambda = 0$, and
$$
\int_{\bR_T \times \bR^d_+} |g|^p \, x_1^{c-2+p} \, dx \, dt < \infty,
\quad
\int_{\bR_T \times \bR^d_+} |f|^p \, x_1^{c-2} \, dx \, dt < \infty.
$$
Then
\begin{align}
							\label{eq0925_3}
&\int_{\bR_T \times \bR^d_+} |u|^p \, x_1^{c-2} \, dx \, dt
\nonumber\\
&\le N \int_{\bR_T \times \bR^d_+} |g|^p \, x_1^{c-2+p} \, dx \, dt
+ N \int_{\bR_T \times \bR^d_+}   |u|^{p-1} |f| \, x_1^c  \, dx \, dt,
\end{align}
where $N = N(d, \delta, c, p)$.
\end{proposition}

\begin{proof}
The proof is similar to that of Proposition \ref{prop0925_1}, but a bit more involved because we need to take care of the terms containing $|Du|^2$  in the right-hand side of inequalities. Moreover, when $a^{11}=a^{11}(x_1)$, it is not possible to multiply both sides of the equation by $1/a^{11}$ as in the non-divergence case.
We prove this case first.

{\bf Case 1}: $a^{11} = a^{11}(x_1)$.
Set
$$
\phi(x_1) = \int_0^{x_1} \frac{1}{a^{11}(r)} \, dr.
$$
Note that $\phi(x_1)$ is comparable to $x_1$, that is
$$
\delta \le \phi(x_1)/x_1 \le \delta^{-1},
\quad
x_1 > 0.
$$

Using $-|u|^{p-2} u \, \left(\phi(x_1)\right)^c$ as a test function on \eqref{eq0925_2}, we have
\begin{multline}
							\label{eq0925_4}
\int_{\bR_T\times \bR^d_+} u_t |u|^{p-2} u \, \phi^c (x_1) \, dx \, dt
+ \int_{\bR_T \times \bR^d_+} a^{ij} D_j u D_i \left( |u|^{p-2} u \, \phi^c (x_1)  \right) \, dx \, dt
\\
+ \lambda \int_{\bR_T \times \bR^d_+} |u|^{p-2}u^2 \, \phi^c (x_1)  \, dx \, dt
= \int_{\bR_T \times \bR^d_+} g_i D_i \left( |u|^{p-2} u \, \phi^c (x_1)  \right) \,  dx \, dt
\\
- \int_{\bR_T \times \bR^d_+} f |u|^{p-2} u \, \phi^c (x_1) \, dx \, dt.
\end{multline}
Note that by integrating by parts
\begin{equation*}
\int_{\bR_T\times \bR^d_+} u_t |u|^{p-2} u \, \phi^c(x_1) \, dx \, dt
= \int_{\bR^d_+} \frac{1}{p} |u|^p(T,x) \,\phi^c(x_1) \, dx,
\end{equation*}
and
\begin{align*}
&\int_{\bR_T \times \bR^d_+} a^{11} D_1 u D_1 \left( |u|^{p-2} u \, \phi^c(x_1) \right) \, dx \, dt\\
&= (p-1) \int_{\bR_T \times \bR^d_+} a^{11} (D_1 u)^2 |u|^{p-2} \, \phi^c(x_1) \, dx \, dt\\
&\quad + c \int_{\bR_T \times \bR^d_+} (D_1 u) |u|^{p-2} u \, \phi^{c-1}(x_1) \, dx \, dt,
\end{align*}
where the last term is equal to
$$
\frac{c(1-c)}{p}\int_{\bR_T \times \bR^d_+} |u|^p \frac{1}{a^{11}} \, \phi^{c-2}(x_1) \, dx \, dt
$$
due to $\frac{1}{p} D_1 \left( |u|^p \right) = |u|^{p-2} u D_1u$ and integration by parts in $x_1$.
For $j=2,\ldots,d$,
\begin{align*}
&\int_{\bR_T \times \bR^d_+} a^{1j} D_j u D_1 \left( |u|^{p-2} u \, \phi^c(x_1) \right) \, dx \, dt\\
&= (p-1) \int_{\bR_T \times \bR^d_+} a^{1j} D_1 u D_j u |u|^{p-2} \, \phi^c(x_1) \, dx \, dt
\end{align*}
since the integral of the term containing the derivative in $x_1$ of $\phi^c(x_1)$ is zero.
Indeed, by integration by parts (note that $a^{ij}$ are independent of $x_j$, $j=2,\ldots,d$)
\begin{align*}
&\int_{\bR_T \times \bR^d_+} \frac{a^{1j}}{a^{11}} D_j u |u|^{p-2} u \, \phi^{c-1}(x_1) \, dx \, dt \\
&=
\frac{1}{p}\int_{\bR_T \times \bR^d_+}
\frac{a^{1j}}{a^{11}} D_j \left( |u|^p \right) \, \phi^{c-1}(x_1) \, dx \, dt = 0.
\end{align*}
For the other $(i,j)$, that is, $i = 2, \ldots, d$,
\begin{align*}
&\int_{\bR_T \times \bR^d_+} a^{ij} D_j u D_i \left( |u|^{p-2} u \, \phi^c (x_1)  \right) \, dx \, dt\\
&= (p-1) \int_{\bR_T \times \bR^d_+} a^{ij} D_i u D_j u |u|^{p-2} \, \phi^c (x_1) \, dx \, dt.
\end{align*}
For the terms in the right-hand side of \eqref{eq0925_4}, we have
\begin{align*}
&\sum_{i=1}^d \int_{\bR_T \times \bR^d_+} g_i D_i \left( |u|^{p-2} u \, \phi^c (x_1)  \right) \,  dx \, dt\\
&= (p-1) \sum_{i=1}^d \int_{\bR_T \times \bR^d_+} g_i |u|^{p-2} D_i u \, \phi^c (x_1) \,  dx \, dt\\
&\quad + c \int_{\bR_T \times \bR^d_+} g_1 |u|^{p-2} u \frac{1}{a^{11}(x_1)} \, \phi^{c-1} (x_1) \,  dx \, dt\\
&\le \varepsilon_1
\int_{\bR_T \times \bR^d_+}
|Du|^2 |u|^{p-2} \, x_1^c \, dx \, dt
+ \varepsilon_2 \int_{\bR_T \times \bR^d_+} |u|^p \, x_1^{c-2} \, dx \, dt\\
&\quad +  N(\varepsilon_1, \varepsilon_2, d, \delta, c, p)
\int_{\bR_T \times \bR^d_+}
|g|^p \, x_1^{c-2+p} \, dx \, dt,
\end{align*}
where we used the condition $p \ge 2$ and Young's inequality,
and
$$
- \int_{\bR_T \times \bR^d_+} f |u|^{p-2} u \, \phi^c (x_1) \, dx \, dt
\le \int_{\bR_T \times \bR^d_+} |u|^{p-1} |f| \phi^c(x_1) \, dx \, dt.
$$

Thus from \eqref{eq0925_4} combined with the above calculations we have
\begin{align}
							\label{eq0925_5}
&\frac{1}{p} \int_{\bR_T \times \bR^d_+} |u|^p(T,x) \, \phi^c (x_1) \, dx \, dt
+ \int_{\bR_T} I(t) \, dt
+ \lambda \int_{\bR_T \times \bR^d_+} |u|^p \, \phi^c (x_1) \, dx \, dt
\nonumber\\
&\le
\varepsilon_1
\int_{\bR_T \times \bR^d_+}
|Du|^2 |u|^{p-2} \, x_1^c \, dx \, dt
+ \varepsilon_2 \int_{\bR_T \times \bR^d_+} |u|^p \, x_1^{c-2} \, dx \, dt
\nonumber\\
&\quad + N
\int_{\bR_T \times \bR^d_+}
|g|^p x_1^{c-2+p} \, dx \, dt +
\int_{\bR_T \times \bR^d_+} |u|^{p-1} |f| \phi^c(x_1) \, dx \, dt,
\end{align}
where $N = N(\varepsilon_1, \varepsilon_2, d, \delta, c, p)$ and
\begin{align*}
I(t) &= (p-1) \int_{\bR^d_+} \frac{a^{ij} + a^{ji}}{2} D_i u D_j u |u|^{p-2} \, \phi^c (x_1)\, dx\\
&\quad + \frac{c(1-c)}{p}\int_{\bR^d_+} |u|^p \frac{1}{a^{11}(x_1)} \, \phi^{c-2} (x_1) \, dx.
\end{align*}

Now, for each $t \in (-\infty, T]$, we consider a change of variables $y=y(t,x)$, where
$$
y_1 = \phi(x_1),
\quad
y_i = - \int_0^{x_1} \frac{a^{1i}(t,r)+a^{i1}(t,r)}{2 a^{11}(r)} \, dr + x_i,
\quad
i = 2, \ldots, d.
$$
Then
$y(0)=0$,
$\partial y_1/\partial x_1 = 1/a^{11}(x_1)$,
$\partial y_i/\partial x_i=1$, $i=2,\ldots,d$,
$$
\partial y_i/\partial x_1 =-\frac{a^{1i}(t,x_1) + a^{i1}(t, x_1)}{2 a^{11}(x_1)},
\quad
i=2, \ldots, d,
$$
and
$$
\partial y_i/\partial x_j =0,
\quad
i \ne j, \,\, j = 2, \ldots, d.
$$
This is a one to one Lipschitz map from $\bR^d_+$ to $\bR^d_+$ and its Jacobian is equal to $1/a^{11}(x_1)$.
Set
$v(t,y) = u(t,x)$.
Then, for each $t \in (-\infty, T]$,
$$
I(t)
= (p-1) \int_{\bR^d_+} \tilde{a}^{kl} D_k v D_l v |v|^{p-2} a^{11}\left(\phi^{-1}(y_1)\right) \, y_1^c \, dy
+ \frac{c(1-c)}{p} \int_{\bR^d_+} |v|^p \, y_1^{c-2} \, dy,
$$
where
$$
\tilde{a}^{kl} = \sum_{i,j=1}^d \frac{a^{ij}+a^{ji}}{2} \frac{\partial y_k}{\partial x_i} \frac{\partial y_l}{\partial x_j}.
$$
By the definition of $y=y(t,x)$ we observe that
$$
\tilde{a}^{11} = \frac{1}{a^{11}\left(\phi^{-1}(y_1)\right)},
\quad
\tilde{a}^{1l} = \tilde{a}^{k1} = 0,
\quad
k, l = 2, \ldots, d,
$$
which implies that
$$
\xi^k \tilde{a}^{kl} \xi^l \ge \delta |\xi|^2
$$
for $\xi \in \bR^d$.
Hence
\begin{align*}
&I(t) = (p-1) \int_{\bR^d_+} (D_1 v)^2 |v|^{p-2} y_1^c \, dy + \frac{c(1-c)}{p} \int_{\bR^d_+} |v|^p y_1^{c-2} \, dy\\
&\,+ (p-1) \sum_{k,l=2}^d \int_{\bR^d_+} \tilde{a}^{kl} D_k v D_l v |v|^{p-2} a^{11}\left( \phi^{-1}(y_1) \right) \, y_1^c \, dy
:= I_1(t) + I_2(t) + I_3(t).
\end{align*}
Set
$$
\nu = \min \left\{ \frac{p-1+c}{2(1-c)}, \frac{p-1}{2} \right\} > 0
$$
and note that, by Lemma \ref{lem0920_1},
\begin{align*}
&I_1(t) + I_2(t)\\
&= (p-1 - \nu) \int_{\bR^d_+} (D_1 v)^2 |v|^{p-2} y_1^c \, dy + \frac{c(1-c)}{p} \int_{\bR^d_+} |v|^p y_1^{c-2} \, dy\\
&\quad+ \nu \int_{\bR^d_+} (D_1 v)^2 |v|^{p-2} y_1^c \, dy\\
&\ge \frac{(1-c)\left(p-1+c - \nu(1-c) \right)}{p^2} \int_{\bR^d_+} |v|^p y_1^{c-2} \, dy
+ \nu \int_{\bR^d_+} (D_1 v)^2 |v|^{p-2} y_1^c \, dy.
\end{align*}
Using the ellipticity condition on $I_3(t)$ and the change of variables back to $u$, we get
\begin{align*}
I(t) &\ge N(c,p) \int_{\bR^d_+} |v|^p y_1^{c-2} \, dy
+\nu \int_{\bR^d_+} |D_1v|^2 |v|^{p-2} y_1^c \, dy\\
&\,\,+ \delta(p-1) \int_{\bR^d_+} |D_{y'}v|^2 |v|^{p-2} a^{11}\left(\phi^{-1}(y_1)\right) \, y_1^c \, dy\\
&\ge N \int_{\bR^d_+} |u|^p x_1^{c-2} \, dx
+ N \int_{\bR^d_+} |D u|^2 |u|^{p-2} x_1^c \, dx,
\end{align*}
where $N =N(\delta,c,p)$.
Then using the above estimate of $I(t)$, the non-negativity of the first and third term in \eqref{eq0925_5},
and appropriate $\varepsilon_1, \varepsilon_2>0$, we finally obtain \eqref{eq0925_3}.

{\bf Case 2}: $a^{11} = a^{11}(t)$.
In this case, we proceed as above with $\phi(x_1)=x_1$.
\end{proof}

Now we give a weighted $L_p$-estimate of $Du$ in terms of those of the lower order term $u$ and the data.
We first show a version of such estimate in Sobolev spaces without weights.

\begin{lemma}
							\label{lem1010}
Let $T \in (-\infty,\infty]$, $\lambda \ge 0$, $1<p<\infty$, and $0 < r < R$.
For $y_1 \in \bR$, set the indicator function
$$
I_\tau(x_1) := 1_{(y_1 - \tau , y_1 + \tau)}(x_1),
\quad
\tau \in (0,\infty)
$$
and $\cD=\bR_T \times (y_1-R, y_1+R)\times \bR^{d-1}$.
If  $u \in \cH_p^1(\cD)$
satisfies
$$
- u_t + D_i \left( a^{ij} D_j u \right) - \lambda u = D_i g_i + f
$$
in $\cD$,
where $g = (g_1,\ldots,g_d)$ and $f, g \in L_p(\cD)$.
Then
\begin{equation}
							\label{eq1010_1}
\sqrt{\lambda}\|I_r u \|_p+\| I_r \, Du \|_p
\\
\le N \|I_R \, g \|_p +
\frac{N(R-r)}{\sqrt{\lambda(R-r)^2+1}} \|I_R f\|_p
+ \frac{N}{R-r} \|I_R \, u\|_p,
\end{equation}
where $\| \cdot \|_p = \| \cdot \|_{L_p\left( (-\infty,T) \times \bR^d \right)}$ and $N = N(d,\delta,p)$.
\end{lemma}

\begin{proof}
Due to translation, we only need to prove the case $y_1 = 0$.
Let $\zeta(x_1)$ be an infinitely differentiable function such that $\zeta(x_1) = 1$ on $x_1 \le 0$ and $\zeta(x_1) = 0$ on $x_1 \ge 1$.
Define $r_0 = r$,
$$
r_n = r + (R-r) \sum_{i = 1}^n 2^{-i},
\quad
\zeta_n(x_1) = \zeta\left( \frac{2^{n+1}}{R-r}\left( |x_1| - r_n \right) \right).
$$
Note that
$\zeta_n(x_1) = 1$ if $|x_1| \le r_n$, $\zeta_n(x_1) = 0$ if $|x_1| \ge r_{n+1}$, and
$$
\left|\zeta'_n(x_1)\right| \le N \frac{2^n}{R-r}.
$$
Set $u_n = u \zeta_n$.
Then we have
$$
- \left(u_n \right)_t
+ D_i \left(a^{ij} D_j u_n \right) - (\lambda + \lambda_n) u_n
= D_i g_{n,i}
+ f_n - \lambda_n u_n
$$
in $\bR_T \times \bR^d$,
where $\lambda_n$, $n=0,1,\ldots$, is an increasing sequence specified below, and
$$
g_{n,i} =  \zeta_n' a^{i1} u +  \zeta_n g_i,
\quad
f_n = \zeta_n' a^{1j} D_j u - \zeta_n' g_1 + \zeta_n f.
$$
By Theorem 5.1 in \cite{MR2764911},
$$
\sqrt{\lambda+\lambda_n} \|u_n\|_p +
\|D u_n \|_p
\le N \left(
\|g_{n,i}\|_p + \frac{1}{\sqrt{\lambda +\lambda_n}}\|f_n\|_p + \frac{\lambda_n}{\sqrt{\lambda +\lambda_n}} \|u_n\|_p \right)
$$
$$
\le N \left(\|g_{n,i}\|_p + \frac{1}{\sqrt{\lambda+\lambda_n}}\|f_n\|_p + \sqrt{\lambda_n} \|u_n\|_p \right),
$$
where $N = N(d,\delta,p)$.
Note that
$$
|I_{n+1} Du| \le |D \left( \zeta_{n+1} u \right)|
= |D u_{n+1}|,
$$
where $I_{n+1} = I_{r_{n+1}}(x_1)$. Then we see that
$$
\| g_{n,i}\|_p
\le N \frac{2^n}{R-r}
\|I_R \, u\|_p
+ N \|I_R \, g_i\|_p,
$$
$$
\| f_n \|_p \le
N \frac{2^n}{R-r} \| D u_{n+1} \|_p
+ N \frac{2^n}{R-r} \|I_R \, g\|_p + N \|I_R \, f\|_p.
$$
Hence
$$
\sqrt{\lambda} \|u_n\|_p + \|D u_n \|_p
\le N_0 \frac{2^n}{R-r} \| I_R u \|_p + N_0 \|I_R g\|_p
+ \frac{N_0 2^n}{\sqrt{\lambda_n}(R-r)}\|D u_{n+1}\|_p
$$
$$
 +  \frac{N_02^n}{\sqrt{\lambda_n}(R-r)} \|I_R g\|_p + \frac{N_0}{\sqrt{\lambda+\lambda_0}} \|I_R f \|_p +N_0 \sqrt{\lambda_n} \|I_R \, u\|_p,
$$
where we take the same $N_0 = N_0(d,\delta,p)\ge 1$ throughout the terms.
Multiply both sides by $\varepsilon^n$ and make summations of both sides with respect to $n = 0,1,\ldots$ to get
\begin{align}
							\label{eq1031_1}
&\sqrt{\lambda} \sum_{n=0}^\infty\varepsilon^n \|u_n\|_p
+ \|D u_0 \|_p
+ \sum_{n=1}^\infty\varepsilon^n \|D u_n \|_p
\nonumber\\
&\le \frac{N_0}{R-r} \sum_{n=0}^\infty (2 \varepsilon)^n\| I_R u \|_p + N_0 \sum_{n=0}^\infty \left( \varepsilon^n + \frac{(2\varepsilon)^n}{\sqrt{\lambda_n} (R-r)}\right)\|I_R g\|_p
\nonumber\\
&\,+ \frac{N_0}{R-r} \sum_{n=1}^\infty \frac{(2\varepsilon)^{n-1}}{\sqrt{\lambda_{n-1}}} \|D u_n\|_p
+ N_0 \sum_{n=0}^\infty\frac{\varepsilon^n}
{\sqrt{\lambda+\lambda_0}} \|I_R f \|_p + N_0 \sum_{n=0}^\infty\varepsilon^n \sqrt{\lambda_n} \|I_R \, u\|_p.
\end{align}
Set
$$
\sqrt{\lambda_n} = \frac{N_0}{R-r} 2^n \varepsilon^{-1},
\quad
n = 0,1,\ldots,
\quad
\varepsilon = \frac 1 8.
$$
Then
$$
\sum_{n=1}^\infty \varepsilon^n \|D u_n\|_p = \frac{N_0}{R-r} \sum_{n=1}^\infty \frac{(2 \varepsilon)^{n-1}}{\sqrt{\lambda_{n-1}}} \| D u_n \|_p < \infty.
$$
After removing the above terms from both sides of \eqref{eq1031_1}, and then calculating the summations,
in particular,
$$
\sum_{n=0}^\infty \frac{\varepsilon^n}{\sqrt{\lambda+\lambda_0}} \|I_R f\|_p \le \frac{2(R-r)}{\sqrt{\lambda (R-r)^2 + 1}} \|I_R f\|_p,
$$
we obtain \eqref{eq1010_1}.
\end{proof}

We establish an a priori estimate in the next proposition.

\begin{proposition}
							\label{pro1220}
Let $T \in (-\infty, \infty]$, $\lambda \ge 0$, $2 \le p < \infty$, $\theta \in (d-1,d-1+p)$, and $u \in C_0^\infty\left((-\infty,T] \times \bR^d_+\right)$ satisfy
\begin{equation}
							\label{eq1004_6}
- u_t + D_i \left( a^{ij} D_j u \right) - \lambda u = D_i g_i + f
\end{equation}
in $\bR_T \times \bR^d_+$,
where $g = (g_1,\ldots,g_d)$,
$g, f \in \bL_{p,\theta}(-\infty,T)$, and $f \equiv 0$ if $\lambda = 0$.
Then
\begin{equation}
							\label{eq1008_1}
\sqrt{\lambda} \|u\|_{p,\theta}
+ \| M^{-1} u \|_{p,\theta} + \| D u \|_{p,\theta}
\le N \|g\|_{p,\theta} + \frac{N}{\sqrt \lambda} \|f\|_{p,\theta},
\end{equation}
where $\| \cdot \|_{p,\theta} = \| \cdot \|_{\bL_{p,\theta}(-\infty,T)}$ and $N = N(d, \delta, \theta, p)$.
\end{proposition}

\begin{proof}
For $r > 0$, set
$$
r_n = 2^{-n/3} r,
\quad
n = -1, 0, 1, \ldots.
$$
Then
$$
r_{n-1}-r_n>r_{n+1}-r_{n+2}=(2^{1/3} - 1) r_{n+2}.
$$
and
$$
(r_{n+1}, r_n)
\subset (r_{n+1},r_{n-1}+r_{n+2}-r_{n+1})
\subset (r_{n+2}, r_{n-1}).
$$
Denote
$$
I_n(x_1) = 1_{(r_{n+1}, r_n)},
\,\,
\tilde{I}_n(x_1) = 1_{(r_{n+1},r_{n-1}+r_{n+2}-r_{n+1})},
\,\,
J_n(x_1) = 1_{(r_{n+2}, r_{n-1})}.
$$
Observe that $u \in \cH_p^{1}\left(\bR_T\times (r_{n+2},r_{n-1}) \times \bR^{d-1}\right)$. By applying Lemma \ref{lem1010} with $\tilde{I}_n$ and $J_n$ in place of $I_r$ and $I_R$, respectively, we get
$$
\lambda^{\frac p 2}\|I_n u\|_p^p+ \| I_n \, Du \|^p_p
\le N \|J_n \, g \|^p_p + N \lambda^{-\frac p 2}\|J_n \, f\|^p_p
+ N r^{-p}_{n+2} \|J_n \, u\|^p_p.
$$
Using the fact that, for $x_1 \in (r_{n+2}, r_{n-1})$,
$$
1 \le \frac{x_1}{r_{n+2}} \le 2
$$
and multiplying both sides by $r_{n+2}^{\theta -d}$,
we obtain
\begin{align*}
&\lambda^{\frac p 2} \int_{\bR_T \times \bR^d_+} I_n |u|^p \, x_1^{\theta-d} \, dx \, dt
+ \int_{\bR_T \times \bR^d_+} I_n |Du|^p \, x_1^{\theta-d} \, dx \, dt\\
&\le N \int_{\bR_T \times \bR^d_+} J_n |g|^p \, x_1^{\theta-d} \, dx \, dt
+ N \lambda^{-\frac p 2} \int_{\bR_T \times \bR^d_+} J_n |f|^p \, x_1^{\theta-d} \, dx \, dt\\
&\,\,+ N \int_{\bR_T \times \bR^d_+} J_n |u|^p \, x_1^{\theta - d - p} \, dx \, dt.
\end{align*}
Take the summations of both sides with respect to $n = 0, 1, \ldots$ to get
\begin{align*}
&\lambda^{\frac p 2} \int_{\bR_T \times \bR^d_+} I_r |u|^p \, x_1^{\theta-d} \, dx \, dt
+ \int_{\bR_T \times \bR^d_+}I_r |Du|^p \, x_1^{\theta - d} \, dx \, dt\\
&\le N \int_{\bR_T \times \bR^d_+} |g|^p \, x_1^{\theta-d} \, dx \, dt
+ N \lambda^{-\frac p 2} \int_{\bR_T \times \bR^d_+}  |f|^p \, x_1^{\theta-d} \, dx \, dt\\
&\,\,+ N \int_{\bR_T \times \bR^d_+} |u|^p \, x_1^{\theta - d - p} \, dx \, dt,
\end{align*}
where we denote
$I_r = 1_{(0,r)}(x_1)$.
Upon sending $r \to \infty$, we see that
\begin{multline}
							\label{eq1219_1}
\lambda^{\frac p 2} \|u\|_{p,\theta}^p + \|Du\|_{p,\theta}^p
\le N \|g\|_{p,\theta}^p + N \lambda^{-\frac p 2} \|f\|_{p,\theta}^p + N\|M^{-1}u\|_{p,\theta}^p
\\
\le N \|g\|_{p,\theta}^p + N \lambda^{-\frac p 2} \|f\|_{p,\theta}^p + N \int_{\bR_T \times \bR^d_+}|u|^{p-1} |f| \, x_1^{\theta-d-p+2} \, dx \, dt,
\end{multline}
where the second inequality is due to Proposition \ref{prop1009} with $c = \theta - d - p +2$.
Observe that, by using Young's inequality twice, the last term in \eqref{eq1219_1} is bounded by
\begin{align*}
&N(\varepsilon) \lambda^{-\frac p 2} \int_{\bR_T \times \bR^d_+} |f|^p \, x_1^{\theta-d} \, dx \, dt
+ \varepsilon \lambda^{\frac{p}{2(p-1)}} \int_{\bR_T \times \bR^d_+} |u|^p \, x_1^{\theta-d-\frac{p-2}{p-1}p} \, dx \, dt\\
&\le N(\varepsilon) \lambda^{-\frac p 2} \|f\|_{p,\theta}^p
+ \varepsilon \left(\|M^{-1}u\|_{p,\theta}^p + \lambda^{\frac p 2} \|u\|_{p,\theta}^p \right)
\end{align*}
for any $\varepsilon > 0$.
From this, \eqref{eq1219_1} with an appropriate $\varepsilon > 0$,  and Hardy's inequality, we prove \eqref{eq1008_1}.
\end{proof}

Now we are ready to prove the main result of this subsection.

\begin{theorem}[Divergence case]
							\label{thm1011_1}
Let $T \in (-\infty, \infty]$, $\lambda \ge 0$, $1 < p < \infty$, $\theta\in (d-1,d-1+p)$, and $u \in \cH_{p,\theta}^{1,\lambda}(-\infty,T)$ satisfy \eqref{eq1004_6} in $\bR_T \times \bR^d_+$,
where $g = (g_1,\ldots,g_d)$,
$g, f \in \bL_{p,\theta}(-\infty,T)$, and $f \equiv 0$ if $\lambda = 0$.
Then we have the estimate \eqref{eq1008_1}.

Moreover, for $g, f \in \bL_{p,\theta}(-\infty,T)$ such that $f \equiv 0$ if $\lambda = 0$, there exists a unique solution $u \in \cH_{p,\theta}^{1, \lambda}(-\infty,T)$ to the equation \eqref{eq1004_6}.
\end{theorem}

\begin{proof}
First we prove the estimate \eqref{eq1008_1} when $p \ge 2$.
Since $-\partial_t+D_i(a^{ij}D_j)-\lambda$ is a bounded operator from $\cH_{p,\theta}^{1,\lambda}(-\infty,T)$ to $M^{-1}\bH^{-1}_{p,\theta}(-\infty,T)+
\sqrt\lambda \bL_{p,\theta}(-\infty,T)$ and $C_0^\infty\left((-\infty,T] \times \bR^d_+\right)$ is dense in $\cH_{p,\theta}^{1,\lambda}(-\infty,T)$,
we find sequences
$$
u_n \in C_0^\infty\left((-\infty,T] \times \bR^d_+\right),
\quad
g_n,\,\,f_n \in \bL_{p,\theta}(-\infty,T)
$$
such that
$$
- {u_n}_t + D_i \left( a^{ij} D_j u_n \right) -\lambda u_n = D_i {g_n}_i + f_n
$$
in $\bR_T \times \bR^d_+$.
Here $f_n\equiv0$ when $\lambda=0$.
Moreover,
$$
\|u_n - u\|_{\cH_{p,\theta}^{1,\lambda}(-\infty,T)} \to 0,
\quad
\|g_n - g\|_{\bL_{p,\theta}(-\infty,T)} \to 0,\quad \|f_n - f\|_{\bL_{p,\theta}(-\infty,T)} \to 0
$$
as $n \to \infty$.
Then we apply Proposition \ref{pro1220} to $u_n$, $g_n$, and $f_n$, and take the limits to get \eqref{eq1008_1}.

As in the non-divergence case (Theorem \ref{thm1011}), to prove the unique solvability, we prove only the solvability of $-u_t + \Delta u - \lambda u = D_i g_i + f$ assuming that $f,g\in C_0^\infty (\bR_T\times \bR^d_+)$.
However, in this case $D_i g_i + f\in C_0^\infty (\bR_T\times \bR^d_+)$, and thus the unique solvability is established in Theorem \ref{thm1011}.

For the case when $p \in (1,2)$, we note that the dual space of $\bL_{p,\theta}(-\infty,T)$ is $\bL_{q, \theta_1}(-\infty,T)$, where
$1/p+1/q=1$ and $\theta/p + \theta_1/q = d$. Keeping this in mind and following the standard duality argument, we obtain the estimate \eqref{eq1008_1} without the term $\|M^{-1}u\|_{p,\theta}$, the estimate of which then follows from Hardy's inequality.
Finally, the solvability is established as above.
The theorem is proved.
\end{proof}

\begin{remark}
							\label{rem12_14}
When $\lambda > 0$ and  $a^{ij}$ satisfy Assumption \ref{assum3.1},
similar to Remark \ref{rem12_9}, we have the following. If $u \in \cH_p^1\left((-\infty,T) \times \bR^d_+\right)$ is a solution to
$$
-u_t + D_i \left(a^{ij} D_j u\right) - \lambda u = D_i g_i + f
$$
in $(-\infty,T) \times \bR^d_+$ with the boundary condition $u(t,0,x') = 0$,
where
$$
g, f \in \bL_{p,\theta}(-\infty,T) \cap L_p\left( (-\infty, T) \times \bR^d_+\right),
$$
then $u$ belongs also to $\cH_{p,\theta}^{1,\lambda}(-\infty,T)$.
To see this, upon approximations using the weighted norm estimate \eqref{eq1008_1} in Theorem \ref{thm1011_1} and the $\cH_p^1$-estimates developed in \cite{MR2764911} for coefficients as in Assumption \ref{assum3.1}, we assume that $a^{ij}$ are infinitely differentiable with bounded derivatives and $g, f \in C_0^\infty\left((-\infty,T] \times \bR^d_+\right)$.
Then $u$ is sufficiently smooth. In particular, by writing the equation into a non-divergence type equation we see $u \in W_p^{1,2}\left((-\infty,T) \times \bR^d_+\right)$, which indicates that, as in the last part of the proof of Theorem \ref{thm1011} (see \eqref{eq1221_4}), $\|M^{-1}u\|_{\bL_{p,\theta}(-\infty,T)} < \infty$. Using this and following the proof of the first inequality in \eqref{eq1219_1}, for which we do not need the assumption $u \in C_0^\infty\left((-\infty,T] \times \bR^d_+\right)$, we conclude that  $u \in \cH_{p,\theta}^{1,\lambda}(-\infty,T)$.
\end{remark}

\section{Mean oscillation estimates}
							\label{sec04}

In this section, we denote $L_0=a^{ij}D_{ij}$ and $\cL_0=D_i(a^{ij}D_j)$, where $a^{ij}$ satisfy Assumption \ref{assum3.1}. The purpose of this section is to establish mean oscillation estimates for the gradient of $u$, which will be used in the subsequent sections.

\subsection{Non-divergence type equations}

By localizing the $L_p$-estimates without weights established in \cite{MR2332574, MR2833589} and using odd/even extensions, we get the following lemma.
\begin{lemma}
                            \label{lem4.1}
Let $p\in (1,\infty)$ and $\lambda\ge 0$. Assume that $u\in W^{1,2}_p(Q_2)$ satisfies $-u_t+L_0u-\lambda u=0$ in $Q_2$. Then there exists a constant $N=N(d,p,\delta)$ such that
$$
\|u\|_{W^{1,2}_p(Q_1)}\le N\|u\|_{L_p(Q_2)}.
$$
The same result holds if we replace $Q$ by $Q^+$ and assume $u=0$ on $\{x_1=0\}$.
\end{lemma}

The next lemma will be used when the center of the cylinder is away from the boundary.

\begin{lemma}
                            \label{lem4.1b}
Let $1<p\le q<\infty$, and $\lambda\ge 0$. Assume that $u\in W^{1,2}_{p}(Q_2)$ satisfies
$-u_t+L_0u-\lambda u=0$ in $Q_2$.
Then $u\in W^{1,2}_q(Q_1)$ and there exists a constant $N=N(d,\delta,p,q)$ such that
\begin{equation}
                            \label{eq10.30}
\|u\|_{W^{1,2}_q(Q_1)}\le N\|u\|_{L_p(Q_2)}.
\end{equation}
Moreover, for $q>d+2$, we have
\begin{equation}
                            \label{eq10.32}
\|Du\|_{C^{\alpha/2,\alpha}(Q_1)}\le N\big\|\sqrt \lambda |u|+|Du|\big\|_{L_p(Q_2)},
\end{equation}
where $\alpha=1-(d+2)/q\in (0,1)$ and $N=N(d,p,q,\delta)$.
\end{lemma}

\begin{proof}
The estimate \eqref{eq10.30} follows from Lemma \ref{lem4.1} by using a standard bootstrap argument. For \eqref{eq10.32}, we first consider the case when $\lambda=0$. In this case, $\tilde u:=u-c$, where $c = \frac{1}{|Q_2|} \int_{Q_2} u \, dx \, dt$, satisfies the same equation as $u$. By \eqref{eq10.30} and Lemma 3.1 of \cite{MR2304157}, we have
$$
\|\tilde u\|_{W^{1,2}_q(Q_1)}\le N\|\tilde u\|_{L_p(Q_2)}\le N\|Du\|_{L_p(Q_2)},
$$
which gives \eqref{eq10.32} by the Sobolev embedding theorem. For the case when $\lambda>0$, we use an idea of S. Agmon; see, for example, \cite[Lemma 6.3.8]{MR2435520}.
Denote by $z = (x,y)$ a point in $\bR^{d+1}$, where $x=(x_1,x') \in \bR^{d}$, $y \in \bR$. We introduce $v(t , z)$ and $\hat{Q}_r$ given by
$$
v(t , z) = v(t, x, y) = u(t, x) \cos(\sqrt{\lambda} y),
$$
$$
\hat{Q}_r = (-r^2,0) \times  (-r,r) \times \{x' \in \bR^{d-1}, y \in \bR : |(x',y)| < r\} .
$$
Observe that
$$
\| D u \|_{C^{\alpha/2, \alpha}(Q_1)}
\le \| D v \|_{C^{\alpha/2, \alpha}(\hat{Q}_1)}.
$$
On the other hand, $v$ satisfies
$$
-v_t+L_0 v + D_y^2v = 0
\quad
\text{in}
\quad
\hat{Q}_{2}.
$$
Then by applying the above result to $v$ we obtain
\begin{equation}								 \label{eq0804}
\| D v \|_{C^{\alpha/2, \alpha}(\hat{Q}_1)}\le N \big\|Dv\big\|_{L_p(\hat Q_2)}.
\end{equation}
Notice that $D v$ is the collection consisting of
$$
\cos( \sqrt{\lambda} y) D_{x} u,
\quad
-\sqrt{\lambda} \sin (\sqrt{\lambda} y ) u.
$$
Thus the right-hand side of \eqref{eq0804} is bounded by the right-hand side of the inequality in \eqref{eq10.32}.
The lemma is proved.
\end{proof}

The lemma below will be used when the center of the cylinder is close to the boundary.

\begin{lemma}
                            \label{lem4.2}
Let $1<p\le q<\infty$, $\theta\in (d-1,d-1+p)$, and $\lambda\ge 0$. Assume that $u\in \fH^2_{p,\theta}(Q_2^+)$ satisfies
$-u_t+L_0 u-\lambda u=0$ in $Q_2^+$. 
Then $u\in W^{1,2}_q(Q_1^+)$ and there exists a constant $N=N(d,\delta,p,q)$ such that
\begin{equation}
                            \label{eq9.01}
\|u\|_{W^{1,2}_q(Q_1^+)}\le N\|u\|_{L_p(Q_2^+)}.
\end{equation}
In particular, if $q>d+2$, we have
\begin{equation}
                            \label{eq9.17}
\|Du\|_{C^{\alpha/2,\alpha}(Q^+_1)}\le N\|u\|_{L_p(Q^+_2)},
\end{equation}
where $\alpha=1-(d+2)/q\in (0,1)$.
\end{lemma}
\begin{proof}
By considering $e^{-\varepsilon t}u$ instead of $u$ and then letting $\varepsilon \to 0$, we may assume that $\lambda>0$.
We take $\zeta\in C_0^\infty((-4,4)\times B_2)$ such that $\zeta\equiv 1$ in $Q_{5/3}$. It is easily seen that $v:=u\zeta\in \fH^2_{p,\theta}(\bR_0\times\bR^{d}_+)$ satisfies
\begin{equation}
                                        \label{eq8.14}
-v_t+L_0 v-\lambda v=-\zeta_t u+2 a^{ij}D_i \zeta D_j u+ a^{ij}D_{ij}\zeta u =: f
\end{equation}
in $\bR_0\times\bR^{d}_+$. Moreover, $f\in \bL_{p,\theta}(\bR_0\times\bR^{d}_+)$. This along with the fact that $f$ has a compact support implies $Mf\in \bL_{p,\theta_1}(\bR_0\times\bR^{d}_+)$ for any $\theta_1\ge \theta -p$. In particular, $M f \in \bL_{p,\theta}(\bR_0 \times \bR^d_+) \cap \bL_{p,d}(\bR_0 \times \bR^d_+)$ because $d \ge \theta - p$.  Then by Theorem \ref{thm1011} the function $v$ also belongs to $\fH^2_{p,d}(\bR_0\times\bR^d_+)$.
To see this, let $\tilde{v} \in \fH^2_{p,d}(\bR_0\times\bR^d_+)$ be a unique solution to \eqref{eq8.14}.
Let $f_n := f I_{x_1 \in (1/n,n)} \in L_p(\bR_0 \times \bR^d_+)$ and $v_n \in W_p^{1,2}(\bR_0 \times \bR^d_+)$ be a unique solution to
$$
-{v_n}_t + L_0 v_n - \lambda v_n = f_n,
\quad
v_n(t,0,x') = 0
$$
in $\bR_0 \times \bR^d_+$,
where, as in Remark \ref{rem12_9},
$v_n \in \fH_{p,\theta}^2(\bR_0 \times \bR^d_+) \cap \fH_{p,d}^2(\bR_0 \times \bR^d_+)$ because $Mf \in \bL_{p,\theta}(\bR_0 \times \bR^d_+) \cap \bL_{p,d}(\bR_0 \times \bR^d_+)$.
Since $M f_n \to M f$ in $\bL_{p,\theta}(\bR_0 \times \bR^d_+)$ and $v \in \fH^2_{p,\theta}(\bR_0\times\bR^d_+)$ is a unique solution to \eqref{eq8.14}, we have $v_n \to v$ in $\fH^2_{p,\theta}(\bR_0\times\bR^d_+)$.
Similarly, $v_n \to \tilde{v}$ in $\fH^2_{p,d}(\bR_0\times\bR^d_+)$.
Hence $v = \tilde{v}$ implying $v \in \fH_{p,d}^2(\bR_0 \times \bR^d_+)$ and $u \in \fH_{p,d}^2(Q_{5/3}^+)$.

Now with a different cutoff function $\eta \in C_0^\infty\left( (-25/9, 25/9) \times B_{5/3} \right)$ such that $\eta \equiv 1$ in $Q_{4/3}$, we define $w := u \eta \in \fH_{p,\theta}^2(\bR_0 \times \bR^d_+)$, which satisfies
$$
- w_t+L_0 w-\lambda w = -\eta_t u + 2a^{ij} D_i \eta D_j u + a^{ij} D_{ij} \eta u =:g
$$
in $\bR_0 \times \bR^d_+$.
As above, $M g \in \bL_{p,\theta}(\bR_0 \times \bR^d_+)$.
Moreover, by the definition of $\eta$ and the fact that $u \in \fH_{p,d}^2(Q_{5/3}^+)$, we see that $g \in L_p(\bR_0 \times \bR^d_+)$.
Therefore, by Remark \ref{rem12_9}, we have $w \in W_p^{1,2}(\bR_0 \times \bR^d_+)$, which implies that $u\in W^{1,2}_p(Q_{4/3}^+)$. Finally, \eqref{eq9.01} follows from Lemma \ref{lem4.1} together with the Sobolev embedding theorem and a standard bootstrap argument, and \eqref{eq9.17} is a simple consequence of \eqref{eq9.01} by using the Sobolev embedding theorem. The lemma is proved.
\end{proof}

For a domain $\cD\subset \bR^{d+1}_+$, denote $(u)_{\cD}$ to be the average of $u$ in $\cD$ with respect to the measure $\mu(dx\,dt)=x_1^{\theta-d}\,dx\,dt$.
Precisely,
$$
\left(u\right)_{\cD}
= \frac{1}{\mu(\cD)}\int_\cD u(t,x) \, \mu(dx \, dt),
\quad
\text{where}
\quad
\mu(\cD) = \int_{\cD}\, \mu(dx \, dt).
$$

\begin{proposition}
                            \label{prop4.3}
Let $p\in (1,\infty)$, $\lambda\ge 0$, $\theta\in (d-1,d-1+p)$, $\alpha\in (0,1)$, $r >0$, $\kappa\ge 32$, and $y_1\ge 0$. Assume that $f \in M^{-1}\bL_{p,\theta}\left(Q_{\kappa r}^+(y_1)\right)$, where $Q_{\kappa r}^+(y_1) = Q_{\kappa r}^+(0,y_1,0)$.
Let $u\in \fH^2_{p,\theta}\left(Q_{\kappa r}^+(y_1)\right)$ be a solution of
\begin{equation*}
-u_t+L_0u-\lambda u=f
\end{equation*}
in $Q_{\kappa r}^+(y_1)$. Then we have
\begin{align}
                                        \label{eq10.57}
\big(|Du-(Du)_{Q^+_r(y_1)}|^p\big)^{1/p}_{Q^+_r(y_1)}\le & N\kappa^{-\alpha}\big((\sqrt \lambda |u|+|Du|)^p )^{1/p}_{Q^+_{\kappa r}(y_1)}\nonumber\\
&\, +N\kappa^{(d+\theta+2)/p}\big(|Mf|^p\big)^{1/p}_{Q^+_{\kappa r}(y_1)},
\end{align}
where $N=N(d,\delta,p,\theta,\alpha)>0$ is a constant.
\end{proposition}
\begin{proof}
Dilations show that it suffices to prove the lemma only for $\kappa r = 8$. We consider two cases.

{\em Case 1:} $y_1\in [0,1]$. Since $r=8/\kappa\le 1/4$, we have
\begin{equation*}
Q_r^+(y_1)\subset Q_2^+\subset Q_4^+\subset Q^+_{\kappa r}(y_1).
\end{equation*}
Thanks to Theorem \ref{thm1011}, there is a unique solution $w\in \fH^2_{p,\theta}(\bR_0\times \bR^d_+)$ to the equation
$$
-w_t+L_0w-\lambda u=f I_{Q_4^+}
$$
in $\bR_0\times \bR^d_+$, and we have
\begin{equation}
                            \label{eq11.11}
\|Dw\|_{\bL_{p,\theta}(\bR_0\times \bR^d_+)}
\le N\|MfI_{Q_4}\|_{\bL_{p,\theta}(\bR_0\times \bR^d_+)}
= N\|Mf\|_{\bL_{p,\theta}(Q_4^+)}.
\end{equation}
Denote $v=u-w\in \fH^2_{p,\theta}(Q_4^+)$, which satisfies
$-v_t+L_0v-\lambda v=0$
in $Q_4^+$. Applying Lemma \ref{lem4.2} to $v$ with a scaling, we get
\begin{align}
                                \label{eq11.35}
&\big(|Dv-(Dv)_{Q^+_r(y_1)}|^p\big)^{1/p}_{Q^+_r(y_1)}
\le Nr^{\alpha}
[Dv]_{C^{\alpha/2,\alpha}(Q^+_2)}\le Nr^{\alpha}\|v\|_{L_p(Q_4^+)}
\nonumber\\
&\,\le N  r^\alpha \|M^{-1}v\|_{\bL_{p,\theta}(Q_4^+)}\le Nr^{\alpha}\|D_1v\|_{\bL_{p,\theta}(Q_4^+)},
\end{align}
where in the third inequality we used the fact that $\theta-d-p<0$ and in the last inequality we used Hardy's inequality.
Combining \eqref{eq11.11} and \eqref{eq11.35}, and using the triangle inequality and the fact $\mu\left(Q_r^+(y_1)\right)\ge N r^{\nu+2}$, $\nu = \max\{d,\theta\}$, we reach
\begin{align*}
&\big(|Du-(Du)_{Q^+_r(y_1)}|^p\big)^{1/p}_{Q^+_r(y_1)}\\
&\,\le N\big(|Dv-(Dv)_{Q^+_r(y_1)}|^p\big)^{1/p}_{Q^+_r(y_1)}
+N\big(|Dw|^p\big)^{1/p}_{Q^+_r(y_1)}\\
&\,\le Nr^{\alpha} \big(|D_1 v|^p\big)^{1/p}_{Q^+_4}+Nr^{-(\nu+2)/p}\big(|Mf|^p \big)^{1/p}_{Q^+_4}\\
&\,\le Nr^{\alpha} \big(|D_1 u|^p\big)^{1/p}_{Q^+_4}+Nr^{-(\nu+2)/p}\big(|Mf|^p \big)^{1/p}_{Q^+_4}.
\end{align*}
Recalling that $r=8/\kappa$ and $Q_4^+\subset Q^+_{\kappa r}(y_1)$, we obtain \eqref{eq10.57} in this case.

{\em Case 2:} $y_1>1$. Since $r=8/\kappa\le 1/4$, we have
\begin{equation*}
Q_r^+(y_1)=Q_r(y_1)\subset Q_{1/4}(y_1)\subset Q_{1/2}(y_1) \subset Q^+_{\kappa r}(y_1).
\end{equation*}
Since $x_1\in (y_1-1/2, y_1+1/2)$ in $Q_{1/2}(y_1)$, the weighted average is comparable to the average without weights.
In particular,
\begin{equation}
							\label{eq1213_1}
N_1 \le \frac{x_1^{\theta-d}}{\mu\left(Q_{1/2}(y_1)\right)} \le N_2
\end{equation}
for $y_1 > 1$ and $x_1 \in (y_1 - 1/2, y_1 + 1/2)$, where $N_{1,2} = N_{1,2}(d,\theta)$.
As before, thanks to Theorem \ref{thm1011}, there is a unique solution $w\in \fH^2_{p,\theta}(\bR_0\times \bR^d_+)$ to the equation
$$
-w_t+L_0 w-\lambda w=f I_{Q_{1/2}(y_1)}
$$
in $\bR_0\times \bR^d_+$, which satisfies the estimate \eqref{eq1002_2} with $w$ and $fI_{Q_{1/2}(y_1)}$ in place of $u$ and $f$, respectively.
This estimate along with the inequality
$$
\lambda^{p/2} x_1^{\theta-d} \le \lambda^p x_1^{\theta-d+p} + x_1^{\theta-d-p}
$$
for all $\lambda \ge 0$ and $x_1 > 0$ shows that
\begin{equation}
                            \label{eq11.11c}
\big\|\sqrt \lambda |w|+|Dw|\big\|_{\bL_{p,\theta}(\bR_0\times \bR^d_+)}
\le N\|Mf\|_{\bL_{p,\theta}\left(Q_{1/2}(y_1)\right)}.
\end{equation}
Denote $v=u-w\in \fH^2_{p,\theta}(\bR_0\times \bR^d_+)$, which satisfies
$-v_t+L_0v-\lambda v=0$
in $Q_{1/2}(y_1)$. Applying Lemma \ref{lem4.1b} with a scaling to $v$, we get
\begin{align}
                                \label{eq11.35c}
&\big(|Dv-(Dv)_{Q^+_r(y_1)}|^p\big)^{1/p}_{Q^+_r(y_1)}\le Nr^{\alpha}[Dv]_{C^{\alpha/2,\alpha}\left(Q_{1/4}(y_1)\right)}\nonumber\\
&\,\le Nr^{\alpha} \big((\sqrt \lambda |v|+|Dv|)^p\big)^{1/p}_{Q_{1/2}(y_1)},
\end{align}
where for the last inequality we used \eqref{eq1213_1}.
Combining \eqref{eq11.11c} and \eqref{eq11.35c}, and using  the fact
$\mu\left(Q_r^+(y_1)\right)\ge N r^{d+2}\mu\left(Q_8^+(y_1)\right)$, we get \eqref{eq10.57} as in the first case. The proposition is proved.
\end{proof}

\subsection{Divergence type equations}

In this subsection, we shall establish analogous oscillation estimates for divergence type equations. The argument is more involved due to the lack of regularity of $Du$ (or $u/x_1$) with respect to $x_1$ near the boundary.

By localizing the $L_p$-estimates established in \cite{MR2764911} and using odd/even extensions, we get the following lemma (see the proof of Lemma \ref{lem1010}).

\begin{lemma}
                            \label{lem4.1d}
Let $p\in (1,\infty)$ and $\lambda\ge 0$. Assume that $u\in \cH^{1}_p(Q_2)$ satisfies $-u_t+\cL_0u-\lambda u=0$ in $Q_2$. Then there exists a constant $N=N(d,\delta,p)$ such that
$$
\|u\|_{\cH^{1}_p(Q_1)}\le N\|u\|_{L_p(Q_2)}.
$$
The same result holds if we replace $Q$ by $Q^+$ and assume $u=0$ on $\{x_1=0\}$.
\end{lemma}

The next lemma will be used when the center of the cylinder is away from the boundary. It follows from Lemma \ref{lem4.1d} by using the Sobolev embedding theorem, a standard bootstrap argument, and finite-difference approximations.

\begin{lemma}
                            \label{lem4.1bd}
Let $1<p\le q<\infty$, and $\lambda\ge 0$. Assume that $u\in \cH^{1}_{p}(Q_2)$ satisfies
$-u_t+\cL_0u-\lambda u=0$ in $Q_2$.
Then $u\in \cH^{1}_q(Q_1)$ and there exists a constant $N=N(d,\delta,p,q)$ such that
\begin{equation*}
\|u\|_{\cH^{1}_q(Q_1)}\le N\|u\|_{L_p(Q_2)}.
\end{equation*}
Moreover, $D_{x'}u\in \cH^{1}_q(Q_1)$ and for $q>d+2$ we have
\begin{equation*}
\|D_{x'}u\|_{C^{\alpha/2,\alpha}(Q_1)}\le N\|D_{x'}u\|_{L_p(Q_2)},
\end{equation*}
where $\alpha=1-(d+2)/q\in (0,1)$ and $N=N(d,p,q,\delta)$.
\end{lemma}

The lemma below will be used when the center of the cylinder is close to the boundary.

\begin{lemma}
                            \label{lem4.2d}
Let $1<p\le q<\infty$, $\theta\in (d-1,d-1+p)$, and $\lambda\ge 0$. Assume that $u\in \cH^{1,\lambda}_{p,\theta}(Q_2^+)$ satisfies
$-u_t+\cL_0 u-\lambda u=0$ in $Q_2^+$.
Then $u\in \cH^{1}_q(Q_1^+)$ and there exists a constant $N=N(d,\delta,p,q)$ such that
\begin{equation}
                            \label{eq9.01d}
\|u\|_{\cH^{1}_q(Q_1^+)}\le N\|u\|_{L_p(Q_2^+)}.
\end{equation}
Furthermore, $D_{x'}u\in \cH^1_q(Q_1^+)$ and if $q>d+2$ we have
\begin{equation}
                            \label{eq9.17d}
\|D_{x'}u\|_{C^{\alpha/2,\alpha}(Q^+_1)}\le N\|u\|_{L_p(Q^+_2)},
\end{equation}
where $\alpha=1-(d+2)/q\in (0,1)$.
\end{lemma}
\begin{proof}
The proof is similar to that of Lemma \ref{lem4.2}. As there, we may assume that $\lambda>0$. We take $\zeta\in C_0^\infty((-4,4)\times B_2)$ such that $\zeta\equiv 1$ in $Q_{5/3}$. It is easily seen that $v:=u\zeta\in \cH^{1,\lambda}_{p,\theta}(\bR_0\times\bR^{d}_+)$ satisfies
\begin{equation*}
-v_t+\cL_0 v-\lambda v=D_i g_i+f
\end{equation*}
in $\bR_0\times\bR^{d}_+$, where
$$
g_i=a^{ij} u D_j \zeta\in \bL_{p,\theta_1}(\bR_0\times\bR^{d}_+),\quad f=-u\zeta_t+a^{ij}D_i\zeta D_j u \in M^{-1}\bL_{p,\theta_1}(\bR_0\times\bR^{d}_+)
$$
for any $\theta_1\ge \theta -p$.
In particular,
$g, Mf \in \bL_{p,\theta} \cap \bL_{p,d}(\bR_0 \times \bR^d_+)$
and by Remark 5.3 in \cite{MR1708104}, we have $h = (h_1, \ldots, h_d) \in \bL_{p,\theta} \cap \bL_{p,d} \left(\bR_0 \times \bR^d_+\right)$ such that
$$
D_i h_i = D_i g_i + f.
$$
Then we see that $h \in L_p\left(\bR_0 \times \bR^d_+\right)$ and, as shown in Remark \ref{rem12_14}, we have $v \in \cH_p^1\left(\bR_0 \times \bR^d_+\right)$.
Hence $u \in \cH_p^1(Q^+_{5/3})$.
Then \eqref{eq9.01d} follows from Lemma \ref{lem4.1d} together with the Sobolev embedding theorem and a standard bootstrap argument. Finally, \eqref{eq9.17d} is a simple consequence of \eqref{eq9.01d} by noting that $D_{x'}u$ satisfies the same equation as $u$ and using the Sobolev embedding theorem. The lemma is proved.
\end{proof}

\begin{proposition}
                            \label{prop4.3d}
Let $p\in (1,\infty)$, $\lambda\ge 0$, $\theta\in (d-1,d-1+p)$, $r > 0$, $\kappa\ge 32$, $y_1 \ge 0$, and $\alpha\in (0,1)$.
Assume that $g, f\in \bL_{p,\theta}\left(Q_{\kappa r}^+(y_1)\right)$ and $f \equiv 0$ if $\lambda = 0$. Let $u\in \cH^{1,\lambda}_{p,\theta}\left(Q_{\kappa r}^+(y_1)\right)$ be a solution of
\begin{equation*}
-u_t+\cL_0 u-\lambda u=D_i g_i+f
\end{equation*}
in $Q_{\kappa r}^+(y_1)$. Then we have
\begin{align*}
&\big(|D_{x'}u-(D_{x'}u)_{Q^+_r(y_1)}|^p\big)^{1/p}_{Q^+_r(y_1)}
\le N\kappa^{-\alpha}\big(|Du|^p\big)^{1/p}_{Q^+_{\kappa r}(y_1)}\nonumber\\
&\, +N\kappa^{(d+\theta+2)/p}\left( |g|^p + \lambda^{-\frac p 2} |f|^p \right)^{1/p}_{Q^+_{\kappa r}(y_1)},
\end{align*}
where $N=N(d,\delta,p,\theta,\alpha)>0$ is a constant.
\end{proposition}

\begin{proof}
The proof is the same as that of Proposition \ref{prop4.3} by using Theorem \ref{thm1011_1}, and Lemmas \ref{lem4.1bd} and \ref{lem4.2d}. We omit the details.
\end{proof}

To estimate $D_1 u$, we consider the following equation of special type:
\begin{equation}
                                \label{eq1.01}
-u_t+D_1(a^{11}D_1 u)+\Delta_{x'} u-\lambda u=D_i g_i+f,
\end{equation}
where $a^{11}=a^{11}(t)$ or $a^{11}(x_1)$, and $\Delta_{x'}=D_2^2+\ldots+D_d^2$. We denote $U=D_1 u$ when $a^{11}=a^{11}(t)$, and $U=a^{11}D_1 u$ when $a^{11}=a^{11}(x_1)$. The next two lemmas show that $U$ is H\"older continuous in both cases.

\begin{lemma}
                            \label{lem4.9}
Let $1<p<\infty$, $\alpha\in (0,1)$, and $\lambda\ge 0$. Assume that $u\in \cH^{1}_{p}(Q_2)$ satisfies
\begin{equation}
                                \label{eq1.41}
-u_t+D_1(a^{11}D_1 u)+\Delta_{x'} u-\lambda u=0
\end{equation}
in $Q_2$. Then there exists a constant $N=N(d,\delta,p,\alpha)$ such that
\begin{equation}
                                    \label{eq1.43}
[U]_{C^{\alpha/2,\alpha}(Q_1)}\le N\|U\|_{L_p(Q_2)}.
\end{equation}
\end{lemma}
\begin{proof}
In the case when $a^{11}=a^{11}(t)$, we easily seen that $U=D_1u$ satisfies the same equation as $u$. Using finite-difference approximations, from Lemma \ref{lem4.1bd} (or \ref{lem4.1b}) and the Sobolev embedding theorem we get \eqref{eq1.43}.

In the case when $a^{11}=a^{11}(x_1)$, again using finite-difference approximations and Lemma \ref{lem4.1bd}, it is easily seen that $D_t^k D_{x'}^l D_1^i u\in L_q(Q_{3/2})$ for $i=0,1$ and any integers $k,l\ge 0$ and $q\in (1,\infty)$. Therefore, by using the equation \eqref{eq1.41}, we deduce that $U\in W^{1,2}_q(Q_{3/2})$ and it satisfies the non-divergence form equation
$$
-U_t+a^{11}D_1^2 U+\Delta_{x'}U-\lambda U=0
$$
in $Q_{3/2}$, which leads to \eqref{eq1.43} by using Lemma \ref{lem4.1b}.
\end{proof}
\begin{remark}
We note that from the proof above it is clear that $U$ possesses better regularity. In fact, $U\in C^{1,2}$. However, we will not use this in the sequel.
\end{remark}

\begin{lemma}
                            \label{lem4.10}
Let $1<p<\infty$, $\theta\in (d-1,d-1+p)$, $\alpha\in (0,1)$, and $\lambda\ge 0$. Assume that $u\in \cH^{1,\lambda}_{p,\theta}(Q_2^+)$ satisfies \eqref{eq1.41} in $Q_2^+$.
Then $U\in C^{\alpha/2,\alpha}(Q_1^+)$ and there exists a constant $N=N(d,\delta,p,\alpha)$ such that
\begin{equation*}
[U]_{C^{\alpha/2,\alpha}(Q_1^+)}\le N\|u\|_{L_p(Q^+_2)}.
\end{equation*}
\end{lemma}
\begin{proof}
By Lemma \ref{lem4.2d}, we have $u\in \cH^1_p(Q_{3/2}^+)$. By taking the odd extension of $u$ and the even extension of $a^{11}$ with respect to $x_1$, we see that $u \in \cH^1_p(Q_{3/2})$ satisfies \eqref{eq1.41} in $Q_{3/2}$. It follows from Lemmas \ref{lem4.9} and \ref{lem4.1d} that
$$
[U]_{C^{\alpha/2,\alpha}(Q_1^+)}\le N\|U\|_{L_p(Q_{3/2})}\le N\|u\|_{L_p(Q_2)}\le N\|u\|_{L_p(Q_2^+)}.
$$
The lemma is proved.
\end{proof}

\begin{proposition}
                            \label{prop4.12}
Let $p\in (1,\infty)$, $\lambda\ge 0$, $\theta\in (d-1,d-1+p)$, $r > 0$, $\kappa\ge 32$, $y_1 \ge 0$, and $\alpha\in (0,1)$.
Assume that $g, f\in \bL_{p,\theta}\left(Q_{\kappa r}^+(y_1)\right)$ and $f \equiv 0$ if $\lambda = 0$. Let $u\in \cH^{1, \lambda}_{p,\theta}\left(Q_{\kappa r}^+(y_1)\right)$ be a solution of \eqref{eq1.01}
in $Q_{\kappa r}^+(y_1)$. Then we have
\begin{align*}
&\big(|U-(U)_{Q^+_r(y_1)}|^p\big)^{1/p}_{Q^+_r(y_1)}
\le N\kappa^{-\alpha}\big(|U|^p\big)^{1/p}_{Q^+_{\kappa r}(y_1)}\nonumber\\
&\, +N\kappa^{(d+\theta+2)/p}\left( |g|^p + \lambda^{-\frac p 2} |f|^p\right)^{1/p}_{Q^+_{\kappa r}(y_1)},
\end{align*}
where $N=N(d,\delta,p,\theta,\alpha)>0$ is a constant.
\end{proposition}
\begin{proof}
Using Lemmas \ref{lem4.9} and \ref{lem4.10}, we follow the proof of Proposition \ref{prop4.3} with obvious modifications. We omit the details.
\end{proof}

\section{Proof of Theorem \ref{thm1}}
							\label{sec05}

Throughout this section, we denote $L=a^{ij}D_{ij}$ and assume $p\in (1,\infty)$, $\lambda\ge 0$, and $\theta\in (d-1,d-1+p)$.

By the proof of \eqref{eq1004_4}, using Theorem 2.2 of \cite{MR2833589}, a scaling argument, and Hardy's inequality, we reduce the estimate of the $\fH^2_{p,\theta}$ norm of $u$ to that of the $\bL_{p,\theta}$ norm of $Du$.

\begin{lemma}
                                \label{lem5.1}
Let $T \in (-\infty, \infty]$ and $\rho\in (1/2,1)$. Then there exists a positive constant $\varepsilon_0$ depending only on $d$, $\delta$, $p$, and $\theta$ such that, for any $\varepsilon\in (0,\varepsilon_0]$, under Assumption A ($\rho,\varepsilon$) or Assumption A$'$ ($\rho,\varepsilon$) the following holds. Suppose that $u \in \fH_{p,\theta}^2(-\infty,T)$ satisfies
\begin{equation*}
- u_t + Lu -\lambda u = f
\end{equation*}
in $\bR_T \times \bR^d_+$, where $f \in M^{-1}\bL_{p,\theta}(-\infty,T)$.
Then
\begin{equation*}
\lambda \| M u \|_{\bL_{p,\theta}(-\infty,T)}
+ \|u \|_{\fH^2_{p,\theta}(-\infty,T)}
\le N \|Mf\|_{\bL_{p,\theta}(-\infty,T)}+N\| D u \|_{\bL_{p,\theta}(-\infty,T)},
\end{equation*}
where $N = N(d, \delta, \theta, p)$ is a constant.
\end{lemma}

\begin{proof}
We give a sketched proof. Take functions $\zeta\in C_0^\infty(2,3)$ and $\tilde \zeta\in C_0^\infty(1,4)$ such that $\tilde \zeta=1$ on $[2,3]$.
Set
$$
w(t,x) = r^2 u(t/r^2,x/r) \zeta(x_1),\quad a^{ij}_r(t,x) = a^{ij}(t/r^2,x/r).
$$
Then $w \in W_p^{1,2}(\bR_T \times \bR^d)$
satisfies
\begin{equation}
							\label{eq1230_1}
-w_t + a^{ij}_r(t,x) D_{ij} w - r^{-2}\lambda w = f_r
\end{equation}
in $\bR_T\times \bR^d$,
where
\begin{align*}
f_r(t,x) = &f(t/r^2,x/r)\zeta(x_1)+2r \sum_{i=1}^d a_r^{i1}(t,x) D_i u(t/r^2,x/r)\zeta'(x_1)\\
&\quad +r^2 a_r^{11}(t,x) u(t/r^2,x/r)\zeta''(x_1).
\end{align*}
Since $a^{ij}$ satisfy Assumption A $(\rho, \varepsilon)$ or Assumption A$'$ $(\rho, \varepsilon)$, the weighted mean oscillation of $a^{ij}(t,x)$ (in $x$, $x'$, or $(t,x')$) is less than $\varepsilon$ on $Q_R(t,x)$ for $R \le \rho x_1$. Then due to the scaling, the (unweighted) mean oscillation of $a^{ij}_r(t,x)$ on $Q_R(t,x)$ is less than $N(d,\theta)\varepsilon$ whenever $x_1 \in (1,4)$ and $R \le 1/2$.
This along with the fact that the equation \eqref{eq1230_1} is zero for $x_1 \notin (2,3)$ allows us to apply the results in \cite{MR2833589}. To do this, write the equation as
\begin{equation}
                                \label{eq1.30}
-w_t + a_r^{ij}(t,x) D_{ij} w - (r^{-2}\lambda + \lambda_0) w = f_r - \lambda_0 w,
\end{equation}
where $\lambda_0 = \lambda_0(d,\delta,p,0,1/2)$ (in the notation of \cite{MR2833589} $K=0$ and $R_0 = 1/2$) is from Theorems 2.2 and 2.3 in \cite{MR2833589}. Since $w=0$ for $x_1 \notin (2,3)$, we can replace $a_r^{ij}$ in \eqref{eq1.30} by $\tilde a_r^{ij}:=\tilde \zeta a_r^{ij}+(1-\tilde\zeta)\delta_{ij}$ and extend the equation to the whole space.
By the aforementioned theorems in \cite{MR2833589},
there exists $\varepsilon_0 = \varepsilon_0(d,\delta,p,\theta)>0$ such that, for $\varepsilon \in (0,  \varepsilon_0]$,
we have
$$
\lambda \|w\|_p + \|D^2w\|_p
+ \|w_t\|_p
\le N \|f_r\|_p + N \lambda_0 \|w\|_p
$$
for any $\lambda \ge 0$,
where $\| \cdot \|_p = \| \cdot \|_{L_p\left(\bR_T \times \bR^d\right)}$ and $N = N(d,\delta,p)$.
Then by changing the variables $(t/r^2, x/r) \to (t,x)$  and following the same lines as in the proof of Theorem \ref{thm1011}, we get \eqref{eq1004_4}. Note that $\lambda_0$ is incorporated into $N = N(d,\delta,p,\theta)$ in \eqref{eq1004_4}.
Finally, we use Hardy's inequality to bound $\|M^{-1} u\|_{p,\theta}$ by $\|Du\|_{p,\theta}$ on the right-hand side.
\end{proof}

To estimate $\|Du\|_{\bL_{p,\theta}}$, we extend the mean oscillation estimate in Proposition \ref{prop4.3} to equations with partially VMO coefficients.

\begin{lemma}
                            \label{lem5.2}
Let $h>0$, $\rho\in (1/2,1)$, $\varepsilon\in (0,1)$, $R\in (0,\rho h)$, $\kappa\ge 32$, and $\beta,\beta'\in (1,\infty)$ satisfying $1/\beta+1/\beta'=1$. Let $u\in \fH^2_{p,\theta}(-\infty,\infty)$ be compactly supported on $Q_{R}(h)=(-R^2,0)\times (h-R,h+R)\times B_R'(0)$ and $f:=-u_t+L u-\lambda u$.
Then under Assumption A ($\rho,\varepsilon$) or Assumption A$'$ ($\rho,\varepsilon$), for any $r> 0$ and $Y=(s,y) \in \overline{\bR^{d+1}_+}$, we have
\begin{align}
                                        \label{eq3.48}
&\big(|Du-(Du)_{Q^+_r(Y)}|^p\big)^{1/p}_{Q^+_r(Y)}\le N_0\kappa^{-1/2}\big((\sqrt \lambda |u|+|Du|)^p\big)^{1/p}_{Q^+_{\kappa r}(Y)}\nonumber\\
&\,+N_1\kappa^{(d+\theta+2)/p}\varepsilon^{1/(\beta'p)}
\big(|MD^2 u|^{\beta p}\big)_{Q^+_{\kappa r}(Y)}^{1/(\beta p)} +N_0\kappa^{(d+\theta+2)/p}\big(|Mf|^p \big)^{1/p}_{Q^+_{\kappa r}(Y)},
\end{align}
where $N_0=N_0(d,\delta,p,\theta)$ and $N_1=N_1(d,\delta,p,\theta,\beta,\rho)>0$ are constants.
\end{lemma}

\begin{proof}
We only treat the case when Assumption A is satisfied. The case when Assumption A$'$ is satisfied is similar. By scaling, we may assume that $h=1$. Obviously, we may also assume that $Q^+_r(Y)\cap Q_R(1)$ is not empty, which implies that $y_1\in (1-R-r,1+R+r)$. We discuss two cases.

{\em Case 1:} $(1+\kappa/\rho)r\le R/\rho-R$. Since $R < \rho$ we have
\begin{equation}
                                        \label{eq1.46}
y_1>R/\rho-R-r\ge \kappa r/\rho.
\end{equation}
In this case, we take $Q=Q_{\kappa r}(Y) = Q_{\kappa r}^+(Y)$.

{\em Case 2:} $(1+\kappa/\rho)r> R/\rho-R$. This along with $\rho< 1 < \kappa$ shows that
\begin{equation}
                            \label{eq1.31}
\kappa r  > (\rho + \kappa)r/2>R(1-\rho)/2.
\end{equation}
In this case, we take $Q=Q_{R}(1)$. We claim that
\begin{equation}
                                        \label{eq1.40}
\mu(Q^+_{\kappa r}(Y))\ge N(d,\theta,\rho)\mu(Q).
\end{equation}
Since $Q_r^+(Y) \cap Q_R(1) \ne \emptyset$, there is a
$z_1 \in \left( (y_1-r)\vee 0, y_1+r \right) \cap (1-R, 1+R)$.
Then
$$
I:= \left(z_1 - \frac{R(1-\rho)}{4}, z_1 + \frac{R(1-\rho)}{4}\right) \cap (1-R, 1+R) \subset \left( (y_1 - \kappa r) \vee 0, y_1 + \kappa r \right).
$$
because for any $x_1 \in I$, by \eqref{eq1.31},
$$
|x_1 - y_1| \le |x_1 - z_1| + |z_1 - y_1|
\le \frac{R(1-\rho)}{4} + r < \kappa r.
$$
We also see that
$|I| \ge \frac{R(1-\rho)}{4}$.
Then we have
\begin{align*}
&\mu\left(Q_{\kappa r}^+(Y)\right) \ge N(d) (\kappa r)^{d+1}\int_I x_1^{\theta-d} \, dx_1\\
&\ge N(d,\rho) R^{d+2} \min\{ (1-\rho)^{\theta-d}, (1+\rho)^{\theta-d}\}
\ge N(d,\theta,\rho) \mu(Q),
\end{align*}
where we used \eqref{eq1.31}, $R < \rho$, and the fact that $I$ is a subset of  $(1-\rho, 1+\rho)$ as well as $\left( (y_1 - \kappa r) \vee 0, y_1 + \kappa r \right)$. Therefore, \eqref{eq1.40} is proved.

Recall \eqref{eq1.46} and $R\in (0,\rho)$. By Assumption A, in both cases the corresponding mean oscillations of $a^{ij}$ in $Q$ are less than $\varepsilon$.
Let
$$
\bar a^{11}(t) = \dashint_{B} a^{11}(t,x) \,\mu_d( dx),
$$
where $B=B_r(y)$ (or $B_R(1)$) if $Q = Q_r^+(Y)$ (or $Q_R(1)$, respectively).
For $(i,j)\neq (1,1)$, let
$$
\bar a^{ij}(t,x_1) = \dashint_{B'} a^{ij}(t,x_1,y') \, dy',
$$
where $B' = B'_r(y')$ (or $B'_R(0)$) if $Q = Q_r^+(Y)$ (or $Q_R(1)$, respectively).
Define $L_0=\bar a^{ij}D_{ij}$. It is clear that
$$
-u_t+L_0 u-\lambda u=f+(\bar a^{ij}-a^{ij})D_{ij}u=:\tilde f.
$$
It follows from Proposition \ref{prop4.3} with $\alpha = 1/2$ that
\begin{align}
                                \label{eq11.31}
\big(|Du-(Du)_{Q^+_r(Y)}|^p\big)^{1/p}_{Q^+_r(Y)}
\le & N_0\kappa^{-1/2}\left((\sqrt \lambda |u|+|Du|)^p \right)^{1/p}_{Q^+_{\kappa r}(Y)}\nonumber\\
&\, +N_0\kappa^{(d+\theta+2)/p}\left(|M\tilde f|^p\right)^{1/p}_{Q^+_{\kappa r}(Y)},
\end{align}
where $N_0=N_0(d,\delta,p,\theta)$.
By the definition of $ \tilde f$, the triangle inequality, and the fact that $u$ vanishes outside $Q_{R}(1)$, we have
\begin{equation}
                                        \label{eq10.56}
\big(|M\tilde f|^p\big)^{1/p}_{Q^+_{\kappa r}(Y)}
\le \big(|Mf|^p\big)^{1/p}_{Q^+_{\kappa r}(Y)}+\left(|M(\bar a^{ij}-a^{ij})I_{Q_R(1)}D_{ij}u |^p\right)^{1/p}_{Q^+_{\kappa r}(Y)}.
\end{equation}
By H\"older's inequality and \eqref{eq1.40}, the last term on the right-hand side above is bounded by
\begin{align}
                                        \label{eq11.01}
&\big(I_{Q_R(1)}|\bar a-a|^{\beta' p}\big)^{1/(\beta' p)}_{Q^+_{\kappa r}(Y)}\big(|MD^2u |^{\beta p}\big)^{1/(\beta p)}_{Q^+_{\kappa r}(Y)}\nonumber\\
&\,\le \left(\mu(Q)/\mu(Q^+_{\kappa r}(Y))\right)^{1/(\beta' p)}
\big(|\bar a-a|^{\beta' p}\big)^{1/(\beta' p)}_{Q}\big(|MD^2u |^{\beta p}\big)^{1/(\beta p)}_{Q^+_{\kappa r}(Y)}\nonumber\\
&\,\le N_1\varepsilon^{1/(\beta'p)}
\big(|MD^2 u|^{\beta p}\big)^{1/(\beta p)}_{Q^+_{\kappa r}(Y)},
\end{align}
where $N_1 = N_1(d,\delta,p,\theta,\beta,\rho)$.
Combining \eqref{eq11.31}, \eqref{eq10.56}, and \eqref{eq11.01}, we obtain \eqref{eq3.48}. The lemma is proved.
\end{proof}

We recall the maximal function theorem and the Fefferman-Stein theorem.
Let
$$
\cQ=\{Q^+_r(z): z=(t,x) \in \bR^{d+1}_+, r \in (0,\infty)\}.
$$
For a function $g$ defined in $\bR^{d+1}_+$,
the weighted (parabolic) maximal and sharp function of $g$ are given by
$$
\cM  g (t,x) = \sup_{Q \in \cQ, (t,x) \in Q} \dashint_{Q} |g(s,y)| \, \mu(dy \,ds),
$$
$$
g^{\#}(t,x) = \sup_{Q \in \cQ, (t,x) \in Q} \dashint_{Q} |g(s,y) -
(g)_{Q}| \, \mu(dy\,ds).
$$
Then for $\theta>d-1$, we have
$$
\| g \|_{\bL_{p,\theta}(\bR^{d+1}_+)} \le N \| g^{\#} \|_{\bL_{p,\theta}(\bR^{d+1}_+)},
\quad
\| \cM  g \|_{\bL_{p,\theta}(\bR^{d+1}_+)} \le N \| g\|_{\bL_{p,\theta}(\bR^{d+1}_+)},
$$
if $g \in \bL_{p,\theta}(\bR^{d+1}_+)$, where $1 < p < \infty$ and $N = N(d,p,\theta)$.
Indeed, the first of the inequalities above is due to the Fefferman-Stein theorem on sharp functions
and the second one to the Hardy-Littlewood maximal function theorem.

Next we prove the $\fH^2_{p,\theta}$-estimate in the special case that the solution $u$ is supported in a cylinder.

\begin{proposition}
                                \label{prop5.3}
Let $h>0$, $\rho\in (1/2,1)$, $\varepsilon \in (0,\varepsilon_0]$, where $\varepsilon_0$ is from Lemma \ref{lem5.1}, and $R\in (0,\rho h)$. Let $u\in \fH^2_{p,\theta}$ be compactly supported on $Q_R(h)$ and $f:=-u_t+L u-\lambda u$.
Then under Assumption A ($\rho,\varepsilon$) or Assumption A$'$ ($\rho,\varepsilon$), we have
\begin{equation}
							\label{eq2.00}
\lambda \| M u \|_{p,\theta}
+ \|u \|_{\fH^2_{p,\theta}}
\le N_0 \|Mf\|_{p,\theta}  + N_1\varepsilon^{1/(\beta'q)} \|MD^2u\|_{p,\theta},
\end{equation}
where $\|\cdot\|_{p,\theta} = \|\cdot\|_{\bL_{p,\theta}(-\infty,\infty)}$, $N_0 = N_0(d,\delta,p,\theta)$, $N_1 = N_1(d, \delta, p,\theta,\rho)$,
and $q, \beta'$ are positive numbers determined by $p$ and $\theta$.
\end{proposition}

\begin{proof}
Let $\kappa\ge 32$  be a constant to be specified. We fix $q\in (1,p)$ and $\beta\in (1,\infty)$, depending only on $p$ and $\theta$, such that $\beta q<p$ and $\theta<d-1+q$. Let $\beta'=\beta/(\beta-1)$. By applying Lemma \ref{lem5.2} with $q$ in place of $p$, we get the following pointwise estimate
\begin{multline*}
(Du)^\#(Y)\le N_0\kappa^{-1/2}\cM^{1/q}\big(\sqrt\lambda |u|+|Du|\big)^q(Y)\\
+N_1\kappa^{(d+\theta+2)/q}
\varepsilon^{1/(\beta' q)}
\cM^{1/(\beta q)}\big(|MD^2 u|^{\beta q}\big)(Y)+
N_0\kappa^{(d+\theta+2)/q}
\cM^{1/q}\big(|Mf|^{q}\big)(Y)
\end{multline*}
for any $Y \in \bR^{d+1}_+$.
This estimate together with the Fefferman-Stein theorem on sharp functions and the Hardy-Littlewood theorem on maximal functions gives
\begin{align}
                                \label{eq2.28}
&\|Du\|_{p,\theta}\le N\|(Du)^\#\|_{p,\theta}\nonumber\\
&\,\le N_0\kappa^{-\frac 12}\big\|\cM^{\frac 1 q}\big(\sqrt\lambda |u|+|Du|\big)^q\big\|_{p,\theta}
+N_1\kappa^{\frac {d+\theta+2} {q}}\varepsilon^{\frac 1 {\beta' q}}
\big\|\cM^{\frac 1 {\beta q}}|MD^2 u|^{\beta q}\big\|_{p,\theta}\nonumber\\
&\quad+N_0\kappa^{\frac {d+\theta+2} {q}}
\big\|\cM^{\frac 1 q}|Mf|^{q}\big\|_{p,\theta}\nonumber\\
&\,\le N_0\kappa^{-\frac 12}\big\|\sqrt\lambda |u|+|Du|\big\|_{p,\theta}
+N_1\kappa^{\frac {d+\theta+2} {q}}
\varepsilon^{\frac 1 {\beta' q}}\big\|
MD^2 u\big\|_{p,\theta}+N_0\kappa^{\frac {d+\theta+2} q} \big\|Mf\big\|_{p,\theta},
\end{align}
where $N_0=N_0(d,\delta,p,\theta)$ and $N_1=N_1(d,\delta,p,\theta,\rho)>0$. Then by Lemma \ref{lem5.1}, \eqref{eq2.28}, and H\"older's inequality, we have
\begin{align}
&\,\lambda \| M u \|_{p,\theta}
+ \|u \|_{\fH^2_{p,\theta}}
\le N_0 \|Mf\|_{p,\theta}+N_0 \| D u \|_{p,\theta}\nonumber\\
&\le N_0 \kappa^{-\frac 12}\big\|\sqrt\lambda |u|+|Du|\big\|_{p,\theta}
+N_1\kappa^{\frac {d+\theta+2} {q}}
\varepsilon^{\frac 1 {\beta' q}}\big\|
MD^2 u\big\|_{p,\theta}+
N_0 \kappa^{\frac {d+\theta+2} q} \big\|Mf\big\|_{p,\theta}
\nonumber
\\
&\le N_0\kappa^{-\frac 12}\big(\lambda \| M u \|_{p,\theta}
+ \|u \|_{\fH^2_{p,\theta}}\big)
+N_1 \kappa^{\frac {d+\theta+2} {q}}
\varepsilon^{\frac 1 {\beta' q}}\big\|
MD^2 u\big\|_{p,\theta}+ N_0 \kappa^{\frac{d+\theta+2} q} \big\|Mf\big\|_{p,\theta}.
\nonumber
\end{align}
To complete the proof of \eqref{eq2.00}, it suffices for us to choose $\kappa$ sufficiently large depending only on $d$, $\delta$, $p$, and $\theta$.
\end{proof}

We are now in the position to prove Theorem \ref{thm1}.

\begin{proof}[Proof of Theorem \ref{thm1}]
For the first and second assertions, by the method of continuity it is enough to prove the a priori estimate \eqref{eq9.23}.

{\em Case 1: $T=\infty$.}
First we prove the case $b^i = c=0$.
We fix a number $\varepsilon_2>0$ to be specified below depending only on $d$, $\delta$, $p$, and $\theta$.
By Lemma 5.6 in \cite{MR2990037} (see also Lemma 3.3 in \cite{MR2111792}), there exist $\rho = \rho(\varepsilon_2) \in (1/2,1)$ and nonnegative $\eta_k \in C_0^\infty(\bR^{d+1}_+)$, $k=1,2,\ldots$, such that
\begin{equation}
							\label{eq1230_3}
\sum_k \eta_k^p \ge 1,
\quad
\sum_k \eta_k \le N(d),
\quad
\sum_k \left(M |D\eta_k| + M^2 |D^2 \eta_k| + M^2 |D_t \eta_k| \right) \le \varepsilon_2^p
\end{equation}
on $\bR^{d+1}_+$
and, for each $k$, there exist $r>0$ and a point $(t,x) \in \bR^{d+1}_+$ such that $r \le \rho x_1$ and $\text{supp} \, \eta_k \subset Q_r(t,x)$.
Observe that $u_k : = u \eta_k$ satisfies
$$
-{u_k}_t + a^{ij}D_{ij} u_k  - \lambda u_k = f \eta_k + 2 a^{ij} D_i u D_j \eta_k + u a^{ij} D_{ij} \eta_k  - u D_t \eta_k
$$
in $\bR^{d+1}_+$.
Then using a translation and Proposition \ref{prop5.3} with $\varepsilon \in (0,\varepsilon_0]$ there, we get
\begin{align*}
\lambda \|M u_k\|_{p,\theta}
+ \|u_k\|_{\fH_{p,\theta}^2}&\le N_0 \|M f \eta_k\|_{p,\theta}
+ N_0 \|M Du D\eta_k\|_{p,\theta} + N_0\|M u D^2 \eta_k \|_{p,\theta}\\
&\quad + N_0 \|M u D_t \eta_k\|_{p,\theta} + N_1\varepsilon^{1/(\beta'q)} \|MD^2u_k\|_{p,\theta},
\end{align*}
where $N_0=N_0(d,\delta,p,\theta)$, $N_1 = N_1(d,\delta,p,\theta,\rho)$, and $q, \beta'$ are positive numbers determined by $p$ and $\theta$.
From this and the properties of $\eta_k$ in \eqref{eq1230_3}, we obtain
\begin{align*}
&\lambda \|M u\|_{p,\theta}
+ \|u\|_{\fH_{p,\theta}^2}\le N_0 \|M f\|_{p,\theta}
+ N_0 \varepsilon_2 \left(\|Du\|_{p,\theta} + \|M^{-1}u\|_{p,\theta}\right)\\
&\quad
+N_1\varepsilon^{1/(\beta'q)} \big(\|MD^2u\|_{p,\theta}
+\varepsilon_2 \|Du\|_{p,\theta} + \varepsilon_2\|M^{-1}u\|_{p,\theta}\big).
\end{align*}
We now first choose $\varepsilon_2 \in (0,1)$ sufficiently small depending only on  $d$, $\delta$, $p$, and $\theta$ such that $N_0 \varepsilon_2 < 1/3$, then choose $\rho=\rho(\varepsilon_2)\in (1/2,1)$ such that \eqref{eq1230_3} is satisfied, and finally $\varepsilon = \varepsilon(d,\delta,p,\theta,\rho) \in (0,\varepsilon_0]$ so that
$$
N_1 \varepsilon^{1/(\beta' q)} < 1/3.
$$
Then the above inequality implies \eqref{eq9.23}.

To deal with equations with non-trivial $b^i$ and $c$, we write
\begin{equation}
							\label{eq1231_1}
-u_t + a^{ij} D_{ij} u - \lambda u = f - b^i D_i u  - cu.
\end{equation}
By choosing an appropriate $\varepsilon_1 = \varepsilon_1(d,\delta,p,\theta) > 0$ in \eqref{eq9.34} and using the estimate proved above for the case $b^i=c=0$, we again obtain \eqref{eq9.23}.

{\em Case 2: $T<\infty$.} Certainly we may assume that $T=0$ by shifting the $t$-coordinate. Take the even extensions of $u$, $f$, $a^{ij}$, $b^i$, and $c$ with respect to $t=0$. Then $u\in \fH_{p,\theta}^2(-\infty,\infty)$. Let $v\in \fH_{p,\theta}^2(-\infty,\infty)$ be the solution of
\begin{equation}
                            \label{eq4.19}
- v_t + a^{ij} D_{ij} v +b^i D_i v+ c v- \lambda v=f 1_{t<0}
\end{equation}
in $\bR^{d+1}_+$. Then $w:=u-v\in \fH_{p,\theta}^2(-\infty,\infty)$ satisfies
$$
- w_t + a^{ij} D_{ij} w +b^i D_iw+ c w- \lambda w=0
$$
in $(-\infty,0)\times \bR^{d}_+$. By the property of the heat equation and the method of continuity, we conclude that $w=0$ for $t<0$, and thus the a priori estimate for $u$. For more details, see Theorem 6.4.1 in \cite{MR2435520}. Finally, the solvability is evident by the argument above in view of \eqref{eq4.19}.

Now we prove the last assertion. Write the equation as in \eqref{eq1231_1}, and use the estimate proved above as well as the boundedness of the lower order coefficients on $x_1 \in (\sigma,\infty)$ to get
$$
\lambda \|M u\|_{p,\theta}
+ \|u\|_{\fH_{p,\theta}^2}\le N_2 \|M f\|_{p,\theta}
+ N_3 \left(\|M Du\|_{p,\theta} + \|M u\|_{p,\theta}\right),
$$
where $N_2 = N_2(d,\delta,p,\theta)$ and $N_3=N_3(d,\delta,p,\theta, K)$. We use the interpolation inequality
$$
\|M Du\|_{p,\theta}
\le \varepsilon_3 \|M D^2 u\|_{p,\theta} + N(\varepsilon_3) \|M u\|_{p,\theta}.
$$
Upon choosing an appropriate $\varepsilon_3 > 0$, we get
$$
(\lambda - N_4) \|M u\|_{p,\theta} + \frac{1}{2} \|u\|_{\fH_{p,\theta}^2} \le N_2 \|M f\|_{p,\theta},
$$
where $N_4=N_4(d,\delta,p,\theta,K)$.
Finally, we choose a sufficiently large constant $\lambda_0 = \lambda_0(d,\delta,p,\theta,K)$ so that $\lambda - N_4 \ge \lambda/2$ for $\lambda \ge \lambda_0$.
\end{proof}


\section{Proof of Theorem \ref{thm2}}
							\label{sec06}

Throughout this section, we denote $\cL u = D_i(a^{ij}D_j u)$ and assume $p\in (1,\infty)$, $\lambda\ge 0$, and $\theta\in (d-1,d-1+p)$. First we estimate $D_{x'} u$.
Following exactly the proof of Lemma \ref{lem5.2} with Proposition \ref{prop4.3d} in place of Proposition \ref{prop4.3}, we obtain

\begin{lemma}
                            \label{lem6.1}
Let $h>0$, $\rho\in (1/2,1)$, $\varepsilon\in (0,1)$, $R\in (0,\rho h)$, $\kappa\ge 32$, and $\beta,\beta'\in (1,\infty)$ satisfying $1/\beta+1/\beta'=1$. Let $u\in \cH^{1,\lambda}_{p,\theta}(-\infty,\infty)$ be compactly supported on $Q_{R}(h)$ and satisfy
$$
-u_t + \cL u - \lambda u = D_i g_i + f
$$
in $\bR^{d+1}_+$,
where $g=(g_1,\ldots,g_d), f \in \bL_{p,\theta}(-\infty,\infty)$
and $f \equiv 0$ if $\lambda = 0$.
Then under Assumption A ($\rho,\varepsilon$) or Assumption A$'$ ($\rho,\varepsilon$), for any $r > 0$ and $Y=(s,y) \in \overline{\bR^{d+1}_+}$, we have
\begin{align*}
&\big(|D_{x'}u-(D_{x'}u)_{Q^+_r(Y)}|^p\big)^{1/p}_{Q^+_r(Y)}\le N_0\kappa^{-1/2}\left( |Du|^p\right)^{1/p}_{Q^+_{\kappa r}(Y)}\nonumber\\
&\,+N_1\kappa^{(d+\theta+2)/p}
\varepsilon^{1/(\beta'p)}
\left(|Du|^{\beta p}\right)_{Q^+_{\kappa r}(Y)}^{1/(\beta p)} +N_0 \kappa^{(d+\theta+2)/p} \left(|g|^p + \lambda^{-\frac p 2} |f|^p \right)^{1/p}_{Q^+_{\kappa r}(Y)},
\end{align*}
where $N_0=N_0(d,\delta,p,\theta)$ and $N_1 = N_1(d,\delta,p,\theta,\beta,\rho) > 0$ are constants.
\end{lemma}

Using Lemma \ref{lem6.1} and following the proof of Proposition \ref{prop5.3}, we prove the proposition below. Unlike Proposition \ref{prop5.3} we do not choose a specific $\kappa \ge 32$ yet, which will be specified later after we have estimates for both $D_{x'}u$ and $D_1 u$.

\begin{lemma}
\label{lem0102_1}
Let $h>0$, $\rho\in (1/2,1)$, $\varepsilon \in (0,1)$, $R\in (0,\rho h)$, and $\kappa \ge 32$. Let $u\in \cH^{1,\lambda}_{p,\theta}(-\infty,\infty)$ be compactly supported on $Q_R(h)$ and satisfy
$$
-u_t + \cL u - \lambda u = D_i g_i + f
$$
in $\bR^{d+1}_+$,
where $g=(g_1,\ldots,g_d), f \in \bL_{p,\theta}(-\infty,\infty)$
and $f \equiv 0$ if $\lambda = 0$.
Then under Assumption A ($\rho,\varepsilon$) or Assumption A$'$ ($\rho,\varepsilon$), we have
\begin{multline*}
\|D_{x'} u\|_{p,\theta}
\le N_0 \kappa^{-1/2} \|Du\|_{p,\theta}
+ N_0 \kappa^{(d+\theta+2)/q} \left(\|g\|_{p,\theta}
+ \lambda^{-1/2} \|f\|_{p,\theta}\right)
\\
+ N_1 \kappa^{(d+\theta+2)/q}\varepsilon^{1/(\beta'q)} \|Du\|_{p,\theta},
\end{multline*}
where $\|\cdot\|_{p,\theta}=\|\cdot\|_{\bL_{p,\theta}(-\infty,\infty)}$, $N_0 = N_0(d,\delta,p,\theta)$, $N_1 = N_1(d, \delta, p,\theta,\rho) > 0$, and $q, \beta'$ are positive numbers determined by $p$ and  $\theta$.
\end{lemma}

Next we shall estimate $D_1 u$ by using a scaling argument. Theorem \ref{th081201} below is from \cite{MR2540989}
and can be considered as a generalized version of the Fefferman-Stein Theorem.
To state the theorem,
let
$$
\bC_l = \{ C_l(i_0, i_1, \ldots, i_d), i_0, i_1, \ldots, i_d \in \bZ ,i_1\ge 0\},
\quad l \in \bZ
$$
be the collection of
partitions given by parabolic dyadic cubes in $\bR^{d+1}_+$
\begin{equation*}
[ i_0 2^{-2l}, (i_0+1)2^{-2l} ) \times [ i_1 2^{-l}, (i_1+1)2^{-l} ) \times \ldots \times [ i_d 2^{-l}, (i_d+1)2^{-l} ).
\end{equation*}
Clearly, $(\bC_l,l\in \bZ)$ and the measure $\mu(dx\,dt)=x_1^{\theta-d}\,dx\,dt$ satisfy the conditions in Definition 2.1 of \cite{MR2540989} provided that $\theta>d-1$.
\begin{theorem}                         \label{th081201}
Let $p \in (1, \infty)$, and $W,V,F\in L_{1,\text{loc}}(\mu,\bR^{d+1}_+)$.
Assume that we have $|W| \le V$
and, for each $l \in \bZ$ and $C \in \bC_l$,
there exists a measurable function $U^C$ on $C$
such that $|W| \le U^C \le V$ on $C$ and
\begin{equation*}                            
\int_C |U^C - \left(U^C\right)_C| \,\mu(dx\,dt)
\le \int_C F(t,x) \,\mu(dx\,dt).
\end{equation*}
Then
$$
\| W\|_{L_p(\mu,\bR^{d+1}_+)}^p
\le N(d,p) \|F\|_{L_p(\mu,\bR^{d+1}_+)}\| V \|_{L_p(\mu,\bR^{d+1}_+)}^{p-1},
$$
provided that $F,V\in L_p(\mu,\bR^{d+1}_+)$.
\end{theorem}

\begin{lemma}
                            \label{lem6.2}
Let $h>0$, $\rho\in (1/2,1)$, $\varepsilon\in (0,1)$, $R\in (0,\rho h)$, and $\beta,\beta'\in (1,\infty)$ satisfying $1/\beta+1/\beta'=1$. Let $u\in \cH^{1,\lambda}_{p,\theta}(-\infty,\infty)$ be compactly supported on $Q_{R}(h)$ and satisfy
$$
-u_t + D_1(a^{11}D_1u) + \Delta_{x'} u - \lambda u = D_i g_i + f
$$
in $\bR^{d+1}_+$,
where $g=(g_1,\ldots,g_d), f \in \bL_{p,\theta}(-\infty,\infty)$
and $f \equiv 0$ if $\lambda = 0$.
Then under Assumption A ($\rho,\varepsilon$) or Assumption A$'$ ($\rho,\varepsilon$),
for each $C \in \bC_l$, $l \in \bZ$, and $\kappa\ge 32$,
there exists a function $U^C$ defined on $C$ such that
\begin{equation}
							\label{eq1231_6}
\delta |D_1 u| \le U^C \le  \delta^{-1} |D_1 u|
\end{equation}
on $C$ and
\begin{equation*}
\left(|U^C-(U^C)_{C}|^p\right)^{1/p}_C\le N_0 \left(F_{0,\kappa}\right)_C + N_1\left(F_{1,\kappa}\right)_C,
\end{equation*}
where $N_0=N(d,\delta,p,\theta)$, $N_1=N(d,\delta,p,\theta,\beta,\rho)> 0$, and
$$
F_{0,\kappa} = \kappa^{-1/2}
\cM^{1/p} \left( |D_1u|^p \right)
+\kappa^{(d+\theta+2)/p}\left(\cM^{1/p}(|g|^p)
 + \lambda^{-1/2} \cM^{1/p} (|f|^p) \right),
$$
$$
F_{1,\kappa} = \kappa^{(d+\theta+2)/p}\varepsilon^{1/(\beta'p)} \cM^{1/(\beta p)} \left(|D_1 u|^{\beta p}\right).
$$
\end{lemma}

\begin{proof}
We only treat the case when Assumption A$'$ is satisfied. The case with Assumption A is simpler. Indeed, in this case we set $U^C := |D_1u|$ regardless of the choice of $C$.

By repeating the proof of Lemma \ref{lem5.2} with Proposition \ref{prop4.12} in place of Proposition \ref{prop4.3}, we obtain
\begin{multline}
							\label{eq1231_7}
\left(|U-(U)_{Q^+_r(Y)}|^p\right)^{1/p}_{Q^+_r(Y)}\le N_0\kappa^{-1/2}\left( |D_1 u|^p\right)^{1/p}_{Q^+_{\kappa r}(Y)}
\\
+N_1\kappa^{(d+\theta+2)/p}\varepsilon^{1/(\beta'p)}
\left(|D_1 u|^{\beta p}\right)_{Q^+_{\kappa r}(Y)}^{1/(\beta p)} + N_0 \kappa^{(d+\theta+2)/p} \left(|g|^p + \lambda^{-\frac p 2} |f|^p \right)^{1/p}_{Q^+_{\kappa r}(Y)}
\end{multline}
for any $r > 0$ and $Y=(s,y) \in \overline{\bR^{d+1}_+}$, where $N_0 = N_0(d,\delta,p,\theta)$ and $N_1=N_1(d,\delta,p,\theta,\beta,\rho)$.
Note that in \eqref{eq1231_7} we have
$$
U = \bar a^{11}(x_1) D_1 u,
\quad
\bar a^{11}(x_1) = \frac{1}{|Q'|}\int_{Q'}a^{11}(s,x_1,z') \,ds\,dz',
$$
where $Q' = \left( (\kappa r)^2 - s, s\right) \times B'_{\kappa r}(y')$ or $Q' = (-R^2,0) \times B'_R(0)$ depending on whether $(1+\kappa/\rho)r \le R/\rho - R$ or $(1+\kappa/\rho)r > R/\rho - R$.

For each $C \in \bC_l$, $l \in \bZ$, we find the smallest $r > 0$ and $Y=(s,y) \in \overline{\bR^{d+1}_+}$ such that $C \subset Q_r^+(Y)$.
On $C$, we set $U^C := |U|$, determined by $Q_r^+(Y)$.
We see that $U^C$ satisfies \eqref{eq1231_6}.
Since $\mu(C)$ is comparable to $\mu(Q_r^+(Y))$, by the triangle inequality and \eqref{eq1231_7}
$$
\left(|U^C-(U^C)_{C}|^p\right)^{1/p}_C\le N(d,\theta) \left(|U^C-(U^C)_{Q^+_r(Y)}|^p\right)^{1/p}_{Q^+_r(Y)}
\le N_0 J_0 + N_1 J_1,
$$
where $J_0$ is the sum of the first and third term in the right-hand side of \eqref{eq1231_7} except the constant $N_0$ and $J_1$ is the second term except $N_1$.
By the definition of the maximal functions, we see that
$J_i \le F_{i,\kappa}(X)$, $i=0,1$, for any $X \in C$. In particular,
$J_i \le \left(F_{i,\kappa}\right)_C$, $i=0,1$.
This finishes the proof.
\end{proof}

\begin{lemma}
							\label{lem6.5}
Let $h>0$, $\rho\in (1/2,1)$, $\varepsilon\in (0,1)$, $R\in (0,\rho h)$, and $\kappa \ge 32$. Let $u\in \cH^{1,\lambda}_{p,\theta}(-\infty,\infty)$ be compactly supported on $Q_{R}(h)$ and satisfy
$$
-u_t + D_1(a^{11}D_1u) + \Delta_{x'} u - \lambda u = D_i g_i + f
$$
in $\bR^{d+1}_+$,
where $g=(g_1,\ldots,g_d), f \in \bL_{p,\theta}(-\infty,\infty)$
and $f \equiv 0$ if $\lambda = 0$.
Then under Assumption A ($\rho,\varepsilon$) or Assumption A$'$ ($\rho,\varepsilon$), we have
\begin{multline*}
\|D_1 u\|_{p,\theta}
\le N_0 \kappa^{-1/2} \|D_1u\|_{p,\theta}
+ N_0 \kappa^{(d+\theta+2)/q} \left(\|g\|_{p,\theta}
+ \lambda^{-1/2} \|f\|_{p,\theta}\right)
\\
+ N_1 \kappa^{(d+\theta+2)/q}\varepsilon^{1/(\beta'q)} \|D_1u\|_{p,\theta},
\end{multline*}
where $\|\cdot\|_{p,\theta}=\|\cdot\|_{\bL_{p,\theta}(-\infty,\infty)}$, $N_0 = N_0(d,\delta,p,\theta)$, $N_1 = N_1(d, \delta, p,\theta,\rho) > 0$, and $q, \beta'$ are positive numbers determined by $p, \theta$.
\end{lemma}

\begin{proof}
We fix $q\in (1,p)$ and $\beta\in (1,\infty)$, depending only on $p$ and $\theta$, such that $\beta q<p$ and $\theta<d-1+q$. Let $\beta'=\beta/(\beta-1)$.
Using Theorem \ref{th081201} with
$W=\delta |D_1 u|$, $V = \delta^{-1} |D_1 u|$ and Lemma \ref{lem6.2} with $q$ in place of $p$ there, we obtain the desired estimate.
In particular,
as in the proof of Proposition \ref{prop5.3}, by the Hardy-Littlewood maximal function theorem we have
\begin{align*}
\|F_{1,\kappa}\|_{p,\theta} &= \kappa^{(d+\theta+2)/q}\varepsilon^{1/(\beta'q)} \big\| \cM^{1/(\beta q)} \left(|D_1 u|^{\beta q}\right)\big\|_{p,\theta}\\
&\le \kappa^{(d+\theta+2)/q}\varepsilon^{1/(\beta'q)} \|D_1 u\|_{p,\theta}.
\end{align*}
Similar inequalities hold for $F_{0,\kappa}$ as well.
The lemma is proved.
\end{proof}

We will use the following technical lemma.

\begin{lemma}
							\label{lem6.6}
Let $\chi > 1$ and $\tilde a^{ij}(t,x_1,x') = a^{ij}(t/\chi^2,x_1/\chi,x')$.
If $a^{ij}$ satisfy Assumption A or Assumption A$'$ with $(\rho, \varepsilon)$,
then $\tilde a^{ij}$ satisfy Assumption A or Assumption A$'$ with $(\rho, N\chi \varepsilon)$, where $N$ depends only on $d$.
\end{lemma}

\begin{proof}
When $d=1$, the lemma is obvious by scaling, so we assume that $d \ge 2$. We only give the proof when $i=j=1$ under Assumption A$'$. The other cases are similar.

Denote $a^{11} = a$, $C_{r,R_i}(s,y') = (s-r^2, s) \times C_{R_i}'(y')$, where
$$
C_{R_i}'(y') = \{z'\in \bR^{d-1}\,|\,|z_i-y_i|< R_i,i=2,\ldots,d\}.
$$
We write $C_{r,R}(s,y')$ and $C'_R(y')$ if $R_i = R$ for all $i = 2, \ldots, d$.
For a subset $\cD' \subset \bR \times \bR^{d-1}$, we set
$$
\left( a(\cdot, z_1, \cdot) \right)_{\cD'} = \dashint_{\cD'} a(\tau,z_1,z') \, dz' \, d\tau,
$$
$$
\left[ a(\cdot, z_1, \cdot) \right]_{\cD'}
= \dashint_{\cD'}
\left| a(\tau,z_1,z') - \left( a(\cdot, z_1, \cdot) \right)_{\cD'}\right| \, dz' \, d\tau.
$$
To prove the lemma,
thanks to scaling, in particular, $a(R^2t,Rx)$ satisfies Assumption A or A$'$ for any $R>0$ if $a(t,x)$ does, and shifting the coordinates, it suffices to show that
\begin{equation}
                                \label{eq10.47}
\dashint_{y_1-1}^{\,\,y_1+1}
\left[ a(\cdot, z_1, \cdot) \right]_{C_{1,\chi}}\mu_1(dz_1) \le N(d)\chi\varepsilon
\end{equation}
for any $y_1\ge 1/\rho$, where
$C_{1,\chi} = (-1,0) \times C'_{\chi}(0)$.
First note that by Assumption A$'$
\begin{equation}
							\label{eq0108_1}
\dashint_{y_1-1}^{\,\,y_1+1}
\left[ a(\cdot, z_1, \cdot) \right]_{C_{1,R_i}(0,y')}\mu_1(dz_1) \le N(d)\varepsilon
\end{equation}
for any $y' \in \bR^{d-1}$ and $y_1 \ge 1/\rho$ if
\begin{equation}
							\label{eq0108_2}
\frac{1}{4\sqrt{d-1}} \le R_i \le \frac{1}{\sqrt{d-1}},
\quad
i = 2,\ldots,d.
\end{equation}
Let $k$ be the smallest integer greater than or equal to $\chi\sqrt{d-1}$. We divide $C'_{\chi}(0)$ into sub-cubes of side length $R:=\chi/(2k)\le 1/(2\sqrt{d-1})$: $C'_{\chi}(0)=\bigcup_{i=1}^{(2k)^{d-1}}C'_{(i)}$. Set $C_{(i)} = (-1,0) \times C'_{(i)}$.  Since $R$ satisfies \eqref{eq0108_2}, from \eqref{eq0108_1} we have
\begin{equation}
                        \label{eq10.35}
\dashint_{y_1-1}^{\,\,y_1+1}
\left[ a(\cdot,z_1,\cdot)\right]_{C_{(i)}} \,\mu_1( d z_1) \le N(d)\varepsilon
\end{equation}
for $i=1,\ldots,(2k)^{d-1}$.
For any  sub-cubes $C_{(i)}$ and $C_{(j)}$ which share a common face, we have $C_{(i)} \cup C_{(j)} = C_{1,R_i}(0,y')$ for some $y' \in \bR^{d-1}$ and $R_i$ satisfying \eqref{eq0108_2}. Thus by \eqref{eq0108_1}, we have
$$
\dashint_{y_1-1}^{\,\,y_1+1}
\left| \left(a(\cdot, z_1, \cdot) \right)_{C_{(i)}} - \left(a(\cdot, z_1, \cdot) \right)_{C_{(j)}}\right| \mu(d z_1) \le N(d) \varepsilon,
$$
which implies that
\begin{equation}
							\label{eq10.46}
\dashint_{y_1-1}^{\,\,y_1+1}
\left| \left(a(\cdot, z_1, \cdot) \right)_{C_{(i)}} - \left(a(\cdot, z_1, \cdot) \right)_{C_{1,\chi}} \right| \mu(d z_1) \le N(d) k \varepsilon
\end{equation}
for $i=1,\ldots,(2k)^{d-1}$.
Combining \eqref{eq10.35} and \eqref{eq10.46}, we have
\begin{align*}
&\dashint_{y_1-1}^{\,\,y_1+1}
\left[ a(\cdot, z_1, \cdot) \right]_{C_{1,\chi}}\mu_1(dz_1)\\
&= \frac{|C_{(1)}|}{|C_{1,\chi}|}
\sum_{i=1}^{(2k)^{d-1}}\dashint_{y_1-1}^{\,\,y_1+1}
\dashint_{C_{(i)}} \left|a(\tau,z_1, z') - \left( a(\cdot, z_1, \cdot) \right)_{C_{1,\chi}}\right| \, dz' \, d\tau \, \mu_1(dz_1)\\
&\le \frac{|C_{(1)}|}{|C_{1,\chi}|}
\sum_{i=1}^{(2k)^{d-1}}
\left(\dashint_{y_1-1}^{\,\,y_1+1}
[ a(\cdot, z_a, \cdot)]_{C_{(i)}} \mu(d z_1) + I_i\right) \le N(d) \chi \varepsilon,
\end{align*}
where $I_i$ is the left-hand side of \eqref{eq10.46}.
Therefore, \eqref{eq10.47} is proved and so is the lemma.
\end{proof}

In order to estimate $D_1 u$ using Lemma \ref{lem6.5}, we move all the second-order derivatives in $\cL u$ except $D_1(a^{11} D_1 u)$ to the right-hand side of the equation. To bound the terms involving $D_1 u$ which appear on the right-hand side of the estimates, we use a scaling argument.

\begin{proposition}
                                \label{prop1231}
Let $h>0$, $\rho\in (1/2,1)$, $\varepsilon \in (0,1)$, and $R\in (0,\rho h)$. Let $u\in \cH^{1,\lambda}_{p,\theta}(-\infty,\infty)$ be compactly supported on $Q_R(h)$ and
satisfy
$$
-u_t + \cL u - \lambda u = D_i g_i + f
$$
in $\bR^{d+1}_+$, where
$g=(g_1,\ldots,g_d), f \in \bL_{p,\theta}(-\infty,\infty)$ and $f \equiv 0$ if $\lambda = 0$.
Then under Assumption A ($\rho,\varepsilon$) or Assumption A$'$ ($\rho,\varepsilon$), we have
\begin{align*}
&\sqrt{\lambda}\|u\|_{p,\theta}
+ \|M^{-1}u\|_{p,\theta}
+ \|D u\|_{p,\theta}\\
&\le N_0 \left(\|g\|_{p,\theta} + \lambda^{-1/2} \|f\|_{p,\theta}\right)
+ N_1\varepsilon^{1/(\beta'q)} \|Du\|_{p,\theta},
\end{align*}
where $\|\cdot\|_{p,\theta}=\|\cdot\|_{\bL_{p,\theta}(-\infty,\infty)}$, $N_0 = N_0(d,\delta,p,\theta)$, $N_1 = N_1(d, \delta, p,\theta,\rho)$, and $q, \beta'$ are positive numbers determined by $p, \theta$.
\end{proposition}

\begin{proof}
Rewrite the given equation as
$$
-u_t + \Delta u - \lambda u = D_i\left(g_i + D_i u - a^{ij} D_j u \right) + f.
$$
Then by Theorem \ref{thm1011_1} it is enough to prove that
\begin{equation}
							\label{eq1231_9}
\|D u\|_{p,\theta}
\le N_0 \left(\|g\|_{p,\theta} + \lambda^{-1/2} \|f\|_{p,\theta}\right)
+ N_1\varepsilon^{1/(\beta'q)} \|Du\|_{p,\theta}.
\end{equation}

For a number $\chi>1$, set
$$
w(t,x_1,x') = u(t/\chi^2, x_1/\chi,x'),
\quad
\tilde a^{ij}(t,x) = a^{ij}(t/\chi^2,x_1/\chi,x').
$$
Then it is easily seen that $w \in \cH^{1,\lambda}_{p,\theta}(-\infty,\infty)$ satisfies
$$
-w_t + D_1(\tilde a^{11} D_1 w) + \Delta_{x'} w - \lambda/\chi^2 w = D_i \tilde g_i + \tilde f
$$
in $\bR^{d+1}_+$, where
$$
\tilde f = \chi^{-2} f(t/\chi^2,x_1/\chi,x'),
\quad
\tilde g_1 =  \chi^{-1} g_1(t/\chi^2, x_1/\chi,x') - \chi^{-1} \sum_{j=2}^d \tilde a^{1j} D_j w,
$$
$$
\tilde g_i = \chi^{-2}g_i(t/\chi^2,x_1/\chi,x') + D_i w - \chi^{-1} \tilde a^{i1} D_1w - \chi^{-2} \sum_{j=2}^d \tilde a^{ij} D_j w,
\quad
i = 2, \ldots, d.
$$
By Lemma \ref{lem6.6}, the coefficient $\tilde a^{11}$ satisfies Assumption A or Assumption A$'$ with $(\rho, N(d)\chi\varepsilon)$.
Since $w$ has a compact support in $Q_{\chi R}(\chi h)$, by Lemma \ref{lem6.5} it follows that
$$
\|D_1 w\|_{p,\theta}
\le N_0 \kappa_1^{-1/2} \|D_1w\|_{p,\theta}
+ N_0 \kappa_1^{(d+\theta+2)/q} \left(\|\tilde g\|_{p,\theta}
+ \chi \, \lambda^{-1/2} \|\tilde f\|_{p,\theta}\right)
$$
$$
+ N_1 \kappa_1^{(d+\theta+2)/q}\left(N(d,\chi)\varepsilon\right)^{1/(\beta'q)} \|D_1w\|_{p,\theta}
$$
for any $\kappa_1 \ge 32$.
From this, the change of variables back to $u$, and the fact that $\chi > 1$, we have
\begin{multline}
							\label{eq0102_3}
\|D_1u\|_{p,\theta} \le N_0 \kappa_1^{(d+\theta+2)/q} \left(\|g\|_{p,\theta}+ \lambda^{-1/2} \|f\|_{p,\theta}
+ \chi \|D_{x'}u\|_{p,\theta} \right)
\\
+ \left(N_0 \kappa_1^{-1/2} +N_0 \chi^{-1}  \kappa_1^{(d+\theta+2)/q} + N_1 \kappa_1^{(d+\theta+2)/q} \left(N(d,\chi) \varepsilon\right)^{1/(\beta'  q)} \right)
\|D_1u\|_{p,\theta}.
\end{multline}
On the other hand, from Lemma \ref{lem0102_1} it follows that
\begin{multline}
							\label{eq0102_4}
\|D_{x'} u\|_{p,\theta}
\le N_0 \kappa_2^{-1/2} \|Du\|_{p,\theta}
+ N_0 \kappa_2^{(d+\theta+2)/q} \left(\|g\|_{p,\theta}
+ \lambda^{-1/2} \|f\|_{p,\theta}\right)
\\
+ N_1 \kappa_2^{(d+\theta+2)/q}\varepsilon^{1/(\beta'q)} \|Du\|_{p,\theta}
\end{multline}
for any $\kappa_2 \ge 32$.
Multiplying \eqref{eq0102_4} by $2N_0 \kappa_1^{(d+\theta+2)/q}\chi$ and adding it to \eqref{eq0102_3}, we get
\begin{align}
							\label{eq0102_4b}
&\|D_1u\|_{p,\theta}+\kappa_1^{(d+\theta+2)/q}\chi\|D_{x'} u\|_{p,\theta}\nonumber\\
&\le  N_0 (\kappa_1\kappa_2)^{(d+\theta+2)/q}\chi \left(\|g\|_{p,\theta}+ \lambda^{-1/2} \|f\|_{p,\theta}\right)
\nonumber\\
&\,\,+ \left(N_0 \kappa_1^{-1/2} + N_0\chi^{-1} \kappa_1^{(d+\theta+2)/q} + N_1 \kappa_1^{(d+\theta+2)/q} \left(N(d,\chi) \varepsilon\right)^{1/(\beta'  q)} \right)
\|D_1u\|_{p,\theta}\nonumber\\
&\,\,+N_0 \kappa_2^{-1/2}\kappa_1^{(d+\theta+2)/q}\chi \|Du\|_{p,\theta}+ N_1 (\kappa_1\kappa_2)^{(d+\theta+2)/q}\chi \, \varepsilon^{1/(\beta'q)} \|Du\|_{p,\theta}.
\end{align}
Now in \eqref{eq0102_4b} first we choose sufficiently large $\kappa_1 = \kappa_1(d,\delta,p,\theta) > 32$ so that $N_0 \kappa_1^{-1/2} < 1/4$,
then $\chi = \chi(d,\delta,p,\theta) > 1$ so that
$$
N_0 \chi^{-1} \kappa_1^{(d+\theta+2)/q} < 1/4,
$$
and finally $\kappa_2=\kappa_2(d,\delta,p,\theta)>32$ so that
$$
N_0 \kappa_2^{-1/2}\kappa_1^{(d+\theta+2)/q}\chi < 1/4
$$
to get \eqref{eq1231_9}.
The proposition is proved.
\end{proof}

\begin{proof}[Proof of Theorem \ref{thm2}]
As in the proof of Theorem \ref{thm1}, we first prove the assertions (i) and (ii), for which it suffices to prove the a priori estimate \eqref{eq9.23b} for $T = \infty$.
Let $b^i = \hat b^i = c=0$.
Let $\varepsilon_2>0$ be a number to be specified below depending only on $d$, $\delta$, $p$, and $\theta$.
We find $\rho = \rho(\varepsilon_2) \in (1/2,1)$ and non-negative functions $\eta_k \in C_0^\infty(\bR^{d+1}_+)$, $k=1,2,\ldots$, satisfying the properties, in particular, \eqref{eq1230_3}, described in the proof of Theorem \ref{thm1}.
Observe that $u_k : = u \eta_k$ satisfies
\begin{equation}
							\label{eq0102_7}
-{u_k}_t + D_i (a^{ij}D_j u_k)  - \lambda u_k = D_i {g_k}_i + f_k
\end{equation}
in $\bR^{d+1}_+$,
where
$$
{g_k}_i = a^{ij} u D_j \eta_k + \eta_k g_i,
\quad
f_k = a^{ij}D_j u D_i \eta_k - g_i D_i \eta_k - u D_t \eta_k + \eta_k f.
$$
Here we note that by \eqref{eq1230_3}
$$
a^{ij}D_j u D_i \eta_k - g_i D_i \eta_k - u D_t \eta_k =: \tilde f_k
\subset M^{-1}\bL_{p,\theta} \subset M^{-1}\bH_{p,\theta}^{-1}.
$$
Then from \eqref{eq0102_5}
there exist $\widetilde {g_k}_i \in \bL_{p,\theta}$ satisfying
$D_i \widetilde {g_k}_i = \tilde f_k$ and
\begin{equation}
							\label{eq0102_8}
\|\widetilde {g_k}_i\|_{p,\theta} \le N \| M \tilde f_k \|_{\bH_{p,\theta}^{-1}}
\le N \| M \tilde f_k \|_{p,\theta},
\end{equation}
where we used $\|\cdot\|_{\bH_{p,\theta}^{-1}} \le \|\cdot\|_{\bL_{p,\theta}}$ in the last inequality.
Hence \eqref{eq0102_7} can be written as
$$
-{u_k}_t + D_i (a^{ij}D_j u_k)  - \lambda u_k = D_i ({g_k}_i + \widetilde {g_k}_i)+ \eta_k f.
$$
Using a translation of the coordinates and Proposition \ref{prop1231}, we get
\begin{align*}
&\sqrt{\lambda}\|u_k\|_{p,\theta}
+ \|M^{-1}u_k\|_{p,\theta}
+ \|D u_k\|_{p,\theta}
\le N_1\varepsilon^{1/(\beta'q)} \|Du_k\|_{p,\theta}\\
&\quad + N_0 \left(\|g_k\|_{p,\theta} + \|\widetilde{g_k}\|_{p,\theta}+ \lambda^{-1/2} \|\eta_k f\|_{p,\theta} \right),
\end{align*}
where $N_0=N_0(d,\delta,p,\theta)$, $N_1 = N_1(d,\delta,p,\theta,\rho)$, and
$q, \beta'$ are positive numbers determined by $p, \theta$.
Then from this and \eqref{eq0102_8}, we have
\begin{align*}
&\sqrt{\lambda}\|u_k\|_{p,\theta}
+ \|M^{-1}u_k\|_{p,\theta}
+ \|D u_k\|_{p,\theta}
\le N_0 \big( \|\eta_k g\|_{p,\theta} + \|M D\eta_k M^{-1} u\|_{p,\theta}\\
&\quad+ \|M D \eta_k Du\|_{p,\theta} +
\|M D \eta_k g\|_{p,\theta}
+ \|M^2 D_t \eta_k M^{-1} u\|_{p,\theta}+ \lambda^{-1/2} \|\eta_k f\|_{p,\theta}\big)\\
&\quad+N_1\varepsilon^{1/(\beta'q)}\|Du_k\|_{p,\theta},
\end{align*}
which together with the properties of $\eta_k$ in \eqref{eq1230_3} implies
\begin{align*}
&\sqrt{\lambda}\|u\|_{p,\theta}
+ \|M^{-1}u\|_{p,\theta}
+ \|D u\|_{p,\theta}
\le N_0 \left( \|g\|_{p,\theta} + \lambda^{-1/2}\|f\|_{p,\theta} \right)\\
&\,+ N_0 \varepsilon_2 \left(\|Du\|_{p,\theta} + \|M^{-1}u\|_{p,\theta}\right)
+N_1\varepsilon^{1/(\beta'q)}\left(\|Du\|_{p,\theta}
+\varepsilon_2\|M^{-1}u\|_{p,\theta}\right).
\end{align*}
As in the non-divergence case, by first choosing an appropriate $\varepsilon_2 \in (0,1)$ so that $N_0\varepsilon_2\le 1/3$ and then $\varepsilon = \varepsilon(d,\delta,p,\theta,\rho) \in (0,1)$ so that $N_1 \varepsilon^{1/(\beta' q)} < 1/3$, we finally prove \eqref{eq9.23b}.
For the assertions (i) and (ii) with non-trivial $b^i$, $\hat b^i$, and $c$, we proceed similarly as in the proof of Theorem \ref{thm1} along with the argument shown above for $\tilde f_k$.

Now we prove the last assertion. Rewrite the equation as
$$
-u_t + D_i(a^{ij} D_j u)  - \lambda u = D_i (g_i - b^i u) - \hat b^i D_i u - cu + f,
$$
and use the estimate proved above and the boundedness of the lower order coefficients on $x_1 \in (\sigma,\infty)$ to get
\begin{align*}
&\sqrt{\lambda}\|u\|_{p,\theta}
+ \|M^{-1}u\|_{p,\theta}
+ \|D u\|_{p,\theta}\\
&\le N_2 \left(\|g\|_{p,\theta} + \|b^i u\|_{p,\theta} + \lambda^{-1/2}\|f - \hat b^i D_i u - c u \|_{p,\theta} \right)\\
&\le N_2 \left( \|g\|_{p,\theta} + \lambda^{-1/2} \|f\|_{p,\theta}\right)
+ N_3 \left( (1+\lambda^{-1/2})\|u\|_{p,\theta} + \lambda^{-1/2}\|Du\|_{p,\theta} \right),
\end{align*}
where $N_2 = N_2(d,\delta,p,\theta)$ and $N_3=N_3(d,\delta,p,\theta, K)$.
Finally we choose a $\lambda_0 = \lambda_0(d,\delta,p,\theta,K)$ so that
$$
N_3 (1+\lambda^{-1/2}) < \sqrt\lambda /2,
\quad
N_3 \lambda^{-1/2} < 1/2
$$
for $\lambda \ge \lambda_0$. The theorem is proved.
\end{proof}

\bibliographystyle{plain}



\def\cprime{$'$}\def\cprime{$'$} \def\cprime{$'$} \def\cprime{$'$}
  \def\cprime{$'$} \def\cprime{$'$}


\end{document}